\documentclass[12pt]{article}
\DeclareUnicodeCharacter{FFFD}{}
\usepackage{amsmath,amssymb,amsthm}
\usepackage{ytableau}
\usepackage{graphicx}
\usepackage{enumitem}
\usepackage{appendix}
\usepackage{tikz}
\usepackage{makecell}
\usepackage[all]{xy}
\usepackage{natbib}
\usepackage[a4paper,top=1in,bottom=1in,left=1.25in,right=1.25in]{geometry}
\usepackage{pifont}

\setlist[enumerate,1]{leftmargin=*}
\setlist[itemize]{leftmargin=1em}

\usepackage{hyperref}
\usepackage[capitalize,nameinlink]{cleveref}
\usepackage{aliascnt}

\theoremstyle{plain}
\newtheorem{theorem}{Theorem}[section]

\newaliascnt{lemma}{theorem}
\newtheorem{lemma}[lemma]{Lemma}
\aliascntresetthe{lemma}

\newaliascnt{proposition}{theorem}
\newtheorem{proposition}[proposition]{Proposition}
\aliascntresetthe{proposition}

\newaliascnt{corollary}{theorem}
\newtheorem{corollary}[corollary]{Corollary}
\aliascntresetthe{corollary}

\theoremstyle{definition}
\newaliascnt{definition}{theorem}
\newtheorem{definition}[definition]{Definition}
\aliascntresetthe{definition}

\newaliascnt{example}{theorem}
\newtheorem{example}[example]{Example}
\aliascntresetthe{example}

\newaliascnt{conjecture}{theorem}
\newtheorem{conjecture}[conjecture]{Conjecture}
\aliascntresetthe{conjecture}

\theoremstyle{remark}
\newaliascnt{remark}{theorem}
\newtheorem{remark}[remark]{Remark}
\aliascntresetthe{remark}

\newaliascnt{question}{theorem}
\newtheorem{question}[question]{Question}
\aliascntresetthe{question}

\crefname{theorem}{Theorem}{Theorems}       \Crefname{theorem}{Theorem}{Theorems}
\crefname{lemma}{Lemma}{Lemmas}             \Crefname{lemma}{Lemma}{Lemmas}
\crefname{proposition}{Proposition}{Propositions} \Crefname{proposition}{Proposition}{Propositions}
\crefname{corollary}{Corollary}{Corollaries} \Crefname{corollary}{Corollary}{Corollaries}
\crefname{definition}{Definition}{Definitions} \Crefname{definition}{Definition}{Definitions}
\crefname{example}{Example}{Examples}       \Crefname{example}{Example}{Examples}
\crefname{remark}{Remark}{Remarks}          \Crefname{remark}{Remark}{Remarks}
\crefname{question}{Question}{Questions}    \Crefname{question}{Question}{Questions}
\crefname{conjecture}{Conjecture}{Conjectures} \Crefname{conjecture}{Conjecture}{Conjectures}

\crefname{equation}{Equation}{Equations}    \Crefname{equation}{Equation}{Equations}
\crefname{figure}{Figure}{Figures}          \Crefname{figure}{Figure}{Figures}
\crefname{table}{Table}{Tables}             \Crefname{table}{Table}{Tables}

\renewenvironment{proof}{\noindent\textbf{Proof.}\ }{\hfill$\square$\par}

\title{Some closed formulas for super Kazhdan-Lusztig polynomials}
\author{\textbf{Shunsuke Hirota}}
\date{\textit{\today}}
\begin{document}

\maketitle
Kazhdan--Lusztig theory for finite-dimensional representations of the
general linear Lie superalgebra $\mathfrak{gl}(m|n)$ is well understood
through the Khovanov arc-diagram combinatorics. On the other hand, the
Kazhdan--Lusztig conjecture for the full super category $\mathcal O$ was
established by Cheng--Lam--Wang and Brundan--Losev--Webster, and the
corresponding Kazhdan--Lusztig polynomials can be computed by Brundan's algorithm.

The aim of this paper is to deepen our understanding of super category
$\mathcal O$ and its Kazhdan--Lusztig polynomials by bringing together
three ingredients: auxiliary modules associated with changes of Borel
subalgebras introduced by Cheng--Lam--Wang, the Koszul grading studied by
Brundan--Losev--Webster, and  Whittaker coinvariants functor investigated
by Brundan--Goodwin. We obtain, in particular, closed formulas for certain
Kazhdan--Lusztig polynomials associated with odd reflections, determine the
socles of some auxiliary modules, and realize certain
pullbacks of standard modules via parabolic Miura transforms.

We also use Khovanov arc-diagram combinatorics to characterize certain
regular dominant weights intrinsically in terms of the underlying abelian
category. We discuss how this viewpoint can be used to obtain concrete examples
relevant to the study of Morita equivalence classes of blocks of super
category $\mathcal O$.

\tableofcontents

\section{Introduction}
\label{sec:intro}

\subsection{Background}

Highest weight representations are basic objects in the representation theory of semisimple
Lie algebras. Among the most fundamental results in the subject is the Kazhdan--Lusztig
conjecture, which asserts that the composition multiplicities of Verma modules are governed by
the combinatorics of the Hecke algebra. This monumental result has had a profound influence on
a broad range of representation theory. The Kazhdan--Lusztig polynomials underlying this
phenomenon continue to be extensively studied from a combinatorial point of view, particularly
in type $A$.

On the other hand, Brundan \cite{brundan2003kazhdan} introduced super Kazhdan--Lusztig polynomials by naturally
extending the Fock space formulation of type $A$ Kazhdan--Lusztig theory. He conjectured that
their values at $q=1$ govern the composition multiplicities of Verma supermodules for the
general linear Lie superalgebra $\mathfrak{gl}(m|n)$, the natural super analogue of
$\mathfrak{gl}_n$, and gave an algorithm for computing these super Kazhdan--Lusztig
polynomials. In the same paper, he proved the corresponding statement for finite-dimensional
representations and recovered the general irreducible character formula previously obtained
by Serganova \cite{serganova1996kazhdan}.

The super Kazhdan--Lusztig polynomials arising in finite-dimensional representation theory
are by now very well understood, especially through the language of the Khovanov arc diagrams
developed by Brundan and Stroppel \cite{brundanstroppelI}. The general super Kazhdan--Lusztig polynomials considered
in the present paper are substantially more complicated.

Brundan's conjecture was proved by Cheng--Lam--Wang \cite{cheng2015brundan} using a remarkable categorical
equivalence known as super duality \cite{cheng2010irreducible}. More generally, they considered variants of
Kazhdan--Lusztig polynomials and Verma supermodules associated with mutually non-conjugate
Borel subalgebras, a phenomenon specific to Lie superalgebras. Cao--Lam \cite{cao2016inversion} subsequently gave an
algorithm for the corresponding Kazhdan--Lusztig polynomials for arbitrary choices of Borel
subalgebra.

Later, Brundan--Losev--Webster \cite{brundan2017tensor} gave another proof using the theory of the degenerate affine
Hecke algebra, formulated through the uniqueness of tensor product categorifications \cite{losev2015uniqueness}. They
showed that suitable subquotients of highest weight categories for Lie superalgebras are
equivalent to corresponding categories for ordinary Lie algebras. This allowed them to prove
Koszulity and thereby give a representation-theoretic interpretation of the variable $q$ in
the super Kazhdan--Lusztig polynomials. Their theory also provided another proof of super
duality \cite{leonard2020graded}.

The purpose of the present paper is to deepen the understanding of these results, with
particular emphasis on super Kazhdan--Lusztig polynomials and the modules introduced by
Cheng--Lam--Wang. We focus especially on the replacement of mutually non-conjugate Borel
subalgebras, or equivalently on odd reflections, which is one of the characteristic phenomena
of Lie superalgebra representation theory.

Schematically, one may say that super Kazhdan--Lusztig polynomials are independent of the
choice of Borel subalgebra $\mathfrak b$ after specializing $q=1$, whereas their graded
versions do depend on $\mathfrak b$.

Throughout the paper, we pay particular attention to the case $m=n$, namely
$\mathfrak{gl}(n|n)$. Existing results strongly suggest that this case exhibits especially
important and elegant structures.

For example, for $\mathfrak{gl}(3|3)$ one can visualize the collection of Borel subalgebras
with fixed even part as a finite graph whose vertices are Borel subalgebras and whose edges
correspond to odd reflections.

\begin{center}
\begin{tikzpicture}[scale=0.75]

% y=0（共有）
\node (A) at (0,0) {$()$};
\node (B) at (2,0) {$(1)$};
\node (D) at (6,0) {$(21)$};
\node (F) at (8,0) {$(2^2)$};

% plus（y>0）
\node (Cp) at (4, 1) {$(1^2)$};
\node (Ep) at (6, 2) {$(1^3)$};
\node (Gp) at (8, 2) {$(21^2)$};

% minus（y<0）
\node (Cm) at (4,-1) {$(2)$};
\node (Em) at (6,-2) {$(3)$};
\node (Gm) at (8,-2) {$(31)$};

% 辺：plus 側
\draw (A) -- (B);
\draw (B) -- (Cp);
\draw (Cp) -- (D) (Cp) -- (Ep);
\draw (D) -- (F) (D) -- (Gp) (Ep) -- (Gp);

% 辺：minus 側（同型）
\draw (B) -- (Cm);
\draw (Cm) -- (D) (Cm) -- (Em);
\draw (D) -- (Gm) (Em) -- (Gm);

% y=0（鏡側）
\node (AR) at (18,0) {$(3^3)$};
\node (BR) at (16,0) {$(3^22)$};
\node (DR) at (12,0) {$(321)$};
\node (FR) at (10,0) {$(31^2)$};

% plus（鏡側）
\node (CpR) at (14, 1) {$(32^2)$};
\node (EpR) at (12, 2) {$(2^3)$};
\node (GpR) at (10, 2) {$(2^21)$};

% minus（鏡側）
\node (CmR) at (14,-1) {$(3^21)$};
\node (EmR) at (12,-2) {$(3^2)$};
\node (GmR) at (10,-2) {$(32)$};

% 辺：plus（鏡側）
\draw (AR) -- (BR);
\draw (BR) -- (CpR);
\draw (CpR) -- (DR) (CpR) -- (EpR);
\draw (DR) -- (FR) (DR) -- (GpR) (EpR) -- (GpR);

% 辺：minus（鏡側）
\draw (BR) -- (CmR);
\draw (CmR) -- (DR) (CmR) -- (EmR);
\draw (DR) -- (GmR) (EmR) -- (GmR);

% 追加の接続
\draw (Gp) -- (FR) (Gm) -- (FR);
\draw (F) -- (GpR) (F) -- (GmR);
\draw (Gp) -- (GpR) (Gm) -- (GmR);

\end{tikzpicture}
\end{center}

In this picture, there is a distinguished ``cube'' in the middle.
Brundan--Goodwin showed that this cube plays a central role in the
realization of the super Soergel functor via Whittaker coinvariants.
We emphasize, in particular, that a special case of generalization of the auxiliary modules
introduced by Cheng--Lam--Wang appears naturally in this construction and
plays an essential role in the functor.
Note that here the upper triangular Borel subalgebra corresponds to $()$.

\subsection{Contents}

We now describe the contents of the paper in more detail.

In \cref{sec:super-kl-combinatorics}, we review the purely combinatorial theory of super Kazhdan--Lusztig
polynomials and their variants, following the existing literature.

In \cref{subsec:category-o-definitions}, we recall the basic representation theory of $\mathfrak{gl}(m|n)$.

In \cref{subsec:categorification-c-monomial}, using the variants of Verma supermodules introduced by Cheng--Lam--Wang,
together with the Duflo--Serganova functor \cite{gorelik2022duflo}, we prove a result that may be viewed as a
refinement of BGG reciprocity \cref{thm:U-N-reciprocity}. By combining this with the Koszul grading and building on the
work of Cao--Lam \cite{cao2016inversion}, we obtain the following result, which shows that the behavior of super
Kazhdan--Lusztig polynomials under odd reflections can be described by a closed formula.

\begin{theorem}
\label{oddkl}
For \(\mathbf b,\mathbf b'\in L(m,n)\) and \(\lambda\in\Lambda\), one has
\[
  \sum_{r\geq 0}
  \bigl[
    M^{\mathbf b}(\lambda-\rho^{\mathbf b})
    :
    L^{\mathbf b'}(\lambda-\rho^{\mathbf b'})\langle -r\rangle
  \bigr]q^r
  =
  q^{\ell_{\lambda}(\mathbf b,\mathbf b')},
\]
where \(\ell_{\lambda}(\mathbf b,\mathbf b')\) is the combinatorial quantity
defined in \cref{ld}.
\end{theorem}

Closely related combinatorics governing homomorphisms between Verma supermodules associated
with different Borel subalgebras was previously established by the author \cite{hirota2025odd,Hirota_PathSubgroupoids} in the settings of
general Kac--Moody Lie superalgebras and Nichols algebras of diagonal type \cite{andruskiewitsch2017finite}. From this point of
view, the theorem above is in a sense a natural consequence that one should expect, and it is
reasonable to anticipate analogous statements for other Lie superalgebras.

We also slightly extend the framework of Cheng--Lam--Wang for their variant modules and the
corresponding Kazhdan--Lusztig polynomials. This allows us to treat in a unified language an
important class of modules considered by Brundan--Goodwin, namely modules that are sent by
the Whittaker coinvariant functor to simple modules over finite $W$-superalgebras.

In \cref{subsec:brundan-goodwin}, we make several remarks concerning these modules and, in particular, prove that they are rigid.

In \cref{subsubsec:kac-wakimoto-weights}, we give a new characterization
of Kac--Wakimoto weights, also known as Kostant, totally connected, or
tame weights.

\begin{theorem}\label{kwww}
Let $\mu\in\Lambda^{\mathrm{reg,dom}}$. Then the following conditions are
equivalent.
\begin{enumerate}
\item
The character of $L(\mu)$ is given by the Kac--Wakimoto character formula.

\item
The set
\[
\left\{
L^{\mathbf b}
\bigl(\mu+\rho^{0^m1^n}-\rho^{\mathbf b}\bigr)
\ \middle|\
\mathbf b\in L(m,n)
\right\}
\]
is precisely the set of composition factors of $K(\mu)$.

\item
We have
$\ell K(\mu)=\operatorname{aty}(\mu)+1$.

\item
We have
\[
\ell\left(
\operatorname{rad}P_{\mathcal F}(\mu)/
\operatorname{rad}^2P_{\mathcal F}(\mu)
\right)
=
\operatorname{aty}(\mu)+1.
\]
\end{enumerate}
\end{theorem}

In particular, combining condition~(4) with the results of
Coulembier--Serganova \cite{coulembier2017homological}, we obtain a
characterization of Kac--Wakimoto weights purely in terms of the
underlying abelian category. Hence the class of Kac--Wakimoto weights is
preserved under equivalences of blocks of super category $\mathcal O$.

One motivation for studying super Kazhdan--Lusztig polynomials is the still open problem of
classifying the Morita equivalence classes of blocks of the super category $\mathcal O$.
Coulembier--Serganova \cite{coulembier2017homological} have made progress from a homological-algebraic perspective, while
Brundan--Goodwin \cite{brundan2019whittaker} have approached the problem using Whittaker coinvariants. Nevertheless, the
problem displays a rather counterintuitive level of difficulty, and only a small part of the
classification is presently understood.
In \cref{subsec:morita-classification-blocks}, we propose a third approach to this problem by \cref{kwww}.
This method should substantially
reduce the amount of computation required. Nevertheless, even after this reduction, the
remaining calculations were sufficiently formidable that we did not pursue the classification
of additional examples. 

In \cref{subsec:parabolic-miura-transforms}, we explain that certain natural modules over finite $W$-superalgebras,
constructed via what we call the parabolic Miura transform, can be interpreted as images of
$\mathfrak b$-Verma supermodules under the Whittaker coinvariant functor.

 Unlike in the classical setting, a dominant projective module is not itself a Verma module. For maximally singular weights we obtain a closed formula for filtration of projectives \cref{subsubsec:maximally-singular-weights}, which may again admits an
interpretation in terms of the graph above. The case of the Kac Wakimoto weight appears to be the most interesting one, but here we only
provide explicit examples \cref{app:gl22-examples,app:gl33-examples}.

\subsection{Conjectures and Questions}

\begin{question}\label{q:H0_Verma_truncated_Yangian}
For an arbitrary Borel subalgebra $\mathfrak b$, can one describe
$H_{0}M^{\mathfrak b}(\lambda)$ purely in terms of the truncated super Yangian?
See \cref{rem:H0_Verma_Miura_limits}.
\end{question}

\begin{question}\label{q:socle_H0_Verma}
What is the socle of $M^{\mathfrak b}(\lambda)$?
What is the socle of $H_{0}M^{\mathfrak b}(\lambda)$?
\end{question}

\begin{question}
Fix a central character. When $\ell M^{()}(\lambda)$ is maximal?
\end{question}

\begin{question}
For $\mathfrak{gl}(n|n)$, can one find any distinguished properties of graphs
$R\langle 0\rangle$ and $r\langle 0\rangle$ introduced in \cref{app:gl22-examples} ?
\end{question}

\begin{question}
Does an analogue of \cref{oddkl} hold for general Kac--Moody Lie superalgebras or
Nichols algebras of diagonal type?
\end{question}

The following conjecture holds for $\mathfrak{gl}(1|1)$.
\begin{conjecture}
\label{conj:highest-weight-structures-from-borels}
Let $\mathcal B$ be a block of super category $\mathcal O$ for
$\mathfrak{gl}(m|n)$. Every highest weight structure on $\mathcal B$ is
equivalent to the highest weight structure induced by some Borel
subalgebra $\mathbf b\in L(m,n)$.
\end{conjecture}

\subsection*{Acknowledgements}

I am grateful to Alexander Sherman for a comment that inspired the proof of \cref{thm:U-N-reciprocity}. I also thank Shun-Jen Cheng for helpful advice and for pointing out several relevant references.

This work was supported by the Japan Society for the Promotion of Science (JSPS) through the Research Fellowship for Young Scientists (DC1), Grant Number JP25KJ1664.

\section{Super Kazhdan--Lusztig combinatorics}
\label{sec:super-kl-combinatorics}

The material in this section follows
\cite{brundan2003kazhdan,brundan2014representations,brundan2017tensor,
cheng2015brundan,cao2016inversion}.
Although the treatment in \cref{subsec:c-monomial-bases} is slightly more
general than that found in the existing literature, it contains no
essentially new results. We therefore omit the proofs.
We use the conventions of
\cite{brundan2014representations,brundan2017tensor}, in which canonical
bases correspond to projective modules. Our Bruhat order is opposite to
that of
\cite{cheng2015brundan,cao2016inversion,brundan2003kazhdan}.

\subsection{Finite Young lattices}
\label{subsec:finite-young-lattices}
\begin{definition}\label{g.ydi}
A \emph{partition} is a weakly decreasing finite sequence of positive integers;
we also allow the empty partition $\varnothing$.
For example, we write the partition $5+3+3+1$ as $(5,3^2,1)$.

We identify partitions with Young diagrams in French notation.
Thus the rows are left-justified, with row lengths weakly decreasing from
bottom to top. A box is assigned coordinates $(i,j)$, where $i$ is the column
number counted from the left and $j$ is the row number counted from the bottom.
\end{definition}

For a word $a$, we write $a^r$ for the concatenation of $r$ copies of $a$.
Thus, for example,
\[
0^m1^n
=
\underbrace{00\cdots 0}_{m}
\underbrace{11\cdots 1}_{n},
\qquad
(01)^n
=
\underbrace{0101\cdots 01}_{n}.
\]

For $m,n\in\mathbb Z_{\geq 0}$, let $L(m,n)$ denote the set of all
$0^m1^n$-sequences, namely all words
\[
\mathbf b=b_1b_2\cdots b_{m+n}
\]
containing exactly $m$ entries equal to $0$ and $n$ entries equal to $1$.

Each $\mathbf b\in L(m,n)$ determines a lattice path from $(n,0)$ to $(0,m)$:
a $1$ corresponds to a unit step to the left, while a $0$ corresponds to a
unit step upward. The region weakly southeast of this path inside the
$m\times n$ rectangle is a Young diagram contained in the rectangle $(n^m)$.
This gives a bijection
\[
L(m,n)
\longleftrightarrow
\{\lambda\mid \lambda\subseteq (n^m)\}.
\]
In what follows, we freely identify a sequence with the corresponding Young
diagram.

\begin{definition}[Edge-colored finite Young lattice]
\label{def:edge-colored-finite-young-lattice}
We regard $L(m,n)$ as an edge-colored graph as follows.
Its color set is
\[
C:=\{1,\ldots,n\}\times\{1,\ldots,m\}.
\]
Two vertices $x,y\in L(m,n)$ are joined by an edge of color $(i,j)$ if and
only if the corresponding Young diagrams differ by the addition or removal
of the box with coordinates $(i,j)$.
\end{definition}

\begin{definition}
Let $G=(V,E,c)$ be a colored simple graph with color set $\mathcal C$, where
$c:E\to\mathcal C$. For a subset $A\subseteq\mathcal C$, set
\[
E(A):=\{e\in E\mid c(e)\in A\}.
\]
Define an equivalence relation $\sim_A$ on $V$ by
\[
u\sim_A v
\quad\Longleftrightarrow\quad
u\text{ and }v\text{ lie in the same connected component of }(V,E(A)).
\]

The \emph{quotient of $G$ by the colors in $A$}, denoted $G/A$, is obtained
by contracting every connected component of $(V,E(A))$ to a single vertex
and retaining all edges whose colors do not belong to $A$.

More explicitly, its vertex set is
\[
V_A:=V/{\sim_A},
\]
and two distinct classes $[x],[y]\in V_A$ are joined whenever there exist
representatives $x'\in[x]$ and $y'\in[y]$ such that
\[
\{x',y'\}\in E
\qquad\text{and}\qquad
c(\{x',y'\})\notin A.
\]
The surviving edge inherits the color of the corresponding edge of $G$.
\end{definition}

\begin{remark}
For every $A\subseteq C$, the quotient $L(m,n)/A$ is naturally an
edge-colored graph, with colors inherited from the edges that survive the
contraction. It has the following useful property:

\[
\text{a walk is shortest if and only if its edge colors are pairwise distinct.}
\]

Equivalently, the distance between two vertices is the number of distinct
colors appearing along any shortest walk between them.

The same combinatorial phenomenon appears in the study of homomorphisms
between Verma modules having the same character but associated with different
Borel subalgebras, both for symmetrizable Kac--Moody Lie superalgebras and
for Nichols algebras of diagonal type; see \cite{hirota2025odd}.
\end{remark}

\subsection{$\mathbf b$-canonical bases}
\label{subsec:b-canonical-bases}

Let $q$ be an indeterminate. We denote by $U_q\mathfrak{sl}_\infty$ the
quantum group over $\mathbb Q(q)$ generated by
$e_i,f_i,k_i^{\pm1}$ for $i\in\mathbb Z$, subject to
\[
k_i k_i^{-1}=k_i^{-1}k_i=1,\qquad k_i k_j=k_j k_i,
\]
\[
k_i e_j k_i^{-1}
=
q^{2\delta_{i,j}-\delta_{i,j+1}-\delta_{i,j-1}}e_j,
\qquad
k_i f_j k_i^{-1}
=
q^{-2\delta_{i,j}+\delta_{i,j+1}+\delta_{i,j-1}}f_j,
\]
\[
e_i f_j-f_j e_i
=
\delta_{i,j}\frac{k_i-k_i^{-1}}{q-q^{-1}},
\]
together with $e_i e_j=e_j e_i$ and $f_i f_j=f_j f_i$ if $|i-j|>1$,
and the quantum Serre relations
\[
e_i^2e_j-(q+q^{-1})e_ie_je_i+e_je_i^2=0,\qquad
f_i^2f_j-(q+q^{-1})f_if_jf_i+f_jf_i^2=0
\]
for $|i-j|=1$.

We equip $U_q\mathfrak{sl}_\infty$ with the Hopf algebra structure determined by
\[
\Delta(e_i)=k_i^{-1}\otimes e_i+e_i\otimes1,\qquad
\Delta(f_i)=1\otimes f_i+f_i\otimes k_i,\qquad
\Delta(k_i)=k_i\otimes k_i,
\]
with $\varepsilon(e_i)=\varepsilon(f_i)=0$ and $\varepsilon(k_i)=1$, and
\[
S(e_i)=-k_i e_i,\qquad
S(f_i)=-f_i k_i^{-1},\qquad
S(k_i)=k_i^{-1}.
\]

Let $V$ be the natural $U_q\mathfrak{sl}_\infty$-module with basis
$\{v_j\mid j\in\mathbb Z\}$. The action is given by
\[
f_i v_j=\delta_{i,j}v_{i+1},\qquad
e_i v_j=\delta_{i+1,j}v_i,\qquad
k_i v_j=q^{\delta_{i,j}-\delta_{i+1,j}}v_j.
\]
We also consider the $U_q\mathfrak{sl}_\infty$-module $W$ with basis
$\{w_j\mid j\in\mathbb Z\}$, on which
\[
f_i w_j=\delta_{i+1,j}w_i,\qquad
e_i w_j=\delta_{i,j}w_{i+1},\qquad
k_i w_j=q^{\delta_{i+1,j}-\delta_{i,j}}w_j.
\]

The assignment $q\mapsto q^{-1}$ defines an involutive automorphism
$c\mapsto\overline c$ of $\mathbb Q(q)$. The \emph{bar involution} on
$U_q\mathfrak{sl}_\infty$ is the $\mathbb Q(q)$-antilinear algebra
automorphism determined by
\[
\overline{e_i}=e_i,\qquad
\overline{f_i}=f_i,\qquad
\overline{k_i}=k_i^{-1}.
\]
Thus $\overline{cu}=\overline c\,\overline u$ for
$c\in\mathbb Q(q)$ and $u\in U_q\mathfrak{sl}_\infty$.

For $r\in\mathbb Z_{\geq0}$, define
\[
[r]:=\frac{q^r-q^{-r}}{q-q^{-1}},
\qquad
[r]!:=\prod_{a=1}^r[a],
\]
with $[0]!=1$. The divided powers are
\[
e_i^{(r)}:=\frac{e_i^r}{[r]!},
\qquad
f_i^{(r)}:=\frac{f_i^r}{[r]!}.
\]

Let $\mathbf b=(b_1,\ldots,b_{m+n})\in L(m,n)$. Define
\[
\mathbb T^{\mathbf b}
:=
V^{b_1}\otimes\cdots\otimes V^{b_{m+n}},
\qquad
V^{b_i}:=
\begin{cases}
V, & b_i=0,\\
W, & b_i=1.
\end{cases}
\]
We regard $\mathbb T^{\mathbf b}$ as a
$U_q\mathfrak{sl}_\infty$-module via the iterated coproduct.

For $\lambda=(\lambda_1,\ldots,\lambda_{m+n})\in\mathbb Z^{m+n}$, set
\[
v_\lambda^{\mathbf b}
:=
v_{\lambda_1}^{b_1}\otimes\cdots\otimes
v_{\lambda_{m+n}}^{b_{m+n}},
\qquad
v_j^0:=v_j,\quad v_j^1:=w_j.
\]
We call $v_\lambda^{\mathbf b}$ a \emph{monomial}. The set
$\{v_\lambda^{\mathbf b}\mid\lambda\in\mathbb Z^{m+n}\}$ is called the
\emph{monomial basis} of $\mathbb T^{\mathbf b}$.

\begin{definition}
Fix $\mathbf b=(b_1,\ldots,b_{m+n})\in L(m,n)$. For
$\lambda,\mu\in\mathbb Z^{m+n}$, we write
$\mu\preceq_{\mathbf b}\lambda$ if the following two conditions hold.

First, for every $r\in\mathbb Z$,
\[
\sum_{\substack{1\leq i\leq m+n\\ \mu_i\leq r}}(-1)^{b_i}
=
\sum_{\substack{1\leq i\leq m+n\\ \lambda_i\leq r}}(-1)^{b_i}.
\]
Second, for every $1\leq j\leq m+n$ and $r\in\mathbb Z$,
\[
\sum_{\substack{1\leq i\leq j\\ \mu_i\leq r}}(-1)^{b_i}
\geq
\sum_{\substack{1\leq i\leq j\\ \lambda_i\leq r}}(-1)^{b_i}.
\]
We call $\preceq_{\mathbf b}$ the \emph{$\mathbf b$-Bruhat order} on
$\mathbb Z^{m+n}$.
\end{definition}

We denote by $\widetilde{\mathbb T}^{\mathbf b}$ the completion of
$\mathbb T^{\mathbf b}$ with respect to the linear topology considered in
\cite{cheng2015brundan,cao2016inversion}. Thus its elements are represented
by possibly infinite linear combinations
$\sum_{\lambda\in\mathbb Z^{m+n}}a_\lambda v_\lambda^{\mathbf b}$ with
$a_\lambda\in\mathbb Q(q)$, subject to the corresponding convergence
condition.

We shall also use the subspace
$\widehat{\mathbb T}^{\mathbf b}\subseteq\widetilde{\mathbb T}^{\mathbf b}$
spanned by elements of the form
\[
v_\lambda^{\mathbf b}
+
\sum_{\mu\prec_{\mathbf b}\lambda}
r_{\mu\lambda}(q)v_\mu^{\mathbf b},
\qquad
r_{\mu\lambda}(q)\in\mathbb Q(q).
\]

Let $w_{\mathbf b}$ denote the shuffle corresponding to $\mathbf b$.
For $\lambda\in\mathbb Z^{m+n}$, write
$w_{\mathbf b}\lambda=(\lambda'_1,\ldots,\lambda'_{m+n})$.
We say that $\lambda$ is \emph{$\mathbf b$-dominant} if
\[
\lambda'_m\leq\cdots\leq\lambda'_1,
\qquad
\lambda'_{m+1}\leq\cdots\leq\lambda'_{m+n},
\]
and \emph{$\mathbf b$-typical} if
$\lambda'_i\neq\lambda'_j$ for all $1\leq i\leq m<j\leq m+n$.

Let $[1,m+n]=I_1\sqcup\cdots\sqcup I_d$ be the decomposition into maximal
consecutive intervals on which $\mathbf b$ is constant. For each $t$, let
$\mathfrak S_{I_t}$ be the symmetric group on the positions in $I_t$, and
let $\mathcal H_{I_t}$ be its Iwahori--Hecke algebra. It is generated by
$H_i$ for $i,i+1\in I_t$, subject to
\[
(H_i-q^{-1})(H_i+q)=0,\qquad
H_iH_{i+1}H_i=H_{i+1}H_iH_{i+1},
\]
and $H_iH_j=H_jH_i$ if $|i-j|>1$.

For $x\in\mathfrak S_{I_t}$, set
$H_x:=H_{i_1}\cdots H_{i_r}$ for any reduced expression
$x=s_{i_1}\cdots s_{i_r}$. The bar involution on $\mathcal H_{I_t}$ is the
$\mathbb Q(q)$-antilinear algebra automorphism determined by
$\overline{H_x}=H_{x^{-1}}^{-1}$, or equivalently by
$\overline{H_i}=H_i^{-1}=H_i-(q^{-1}-q)$.

Set
$\mathcal H^{\mathbf b}:=
\mathcal H_{I_1}\otimes\cdots\otimes\mathcal H_{I_d}$,
equipped with the tensor product bar involution
$\overline{h_1\otimes\cdots\otimes h_d}
=\overline{h_1}\otimes\cdots\otimes\overline{h_d}$.

The algebra $\mathcal H^{\mathbf b}$ acts on $\mathbb T^{\mathbf b}$ from
the right. If $i$ and $i+1$ belong to the same interval $I_t$, then
\[
v_\lambda^{\mathbf b}H_i
=
\begin{cases}
v_{\lambda\cdot s_i}^{\mathbf b},
& \lambda\succ_{\mathbf b}\lambda\cdot s_i,\\[2mm]
q^{-1}v_\lambda^{\mathbf b},
& \lambda=\lambda\cdot s_i,\\[2mm]
v_{\lambda\cdot s_i}^{\mathbf b}
-(q-q^{-1})v_\lambda^{\mathbf b},
& \lambda\prec_{\mathbf b}\lambda\cdot s_i.
\end{cases}
\]
Here $s_i=(i,i+1)$ acts on $\mathbb Z^{m+n}$ by interchanging the
$i$-th and $(i+1)$-st entries.

\begin{proposition}
\cite{brundan2003kazhdan,cheng2015brundan}
There exists a $\mathbb Q(q)$-antilinear involution
$\overline{\phantom{x}}:\widehat{\mathbb T}^{\mathbf b}
\to\widehat{\mathbb T}^{\mathbf b}$ with the following properties.
\begin{enumerate}
\item
For every $u\in U_q\mathfrak{sl}_\infty$, $h\in\mathcal H^{\mathbf b}$,
and $v\in\widehat{\mathbb T}^{\mathbf b}$ for which the expression is
defined,
\[
\overline{uvh}=\overline u\,\overline v\,\overline h.
\]

\item
For every $\lambda\in\mathbb Z^{m+n}$,
\[
\overline{v_\lambda^{\mathbf b}}
=
v_\lambda^{\mathbf b}
+
\sum_{\mu\prec_{\mathbf b}\lambda}
r_{\mu\lambda}(q)v_\mu^{\mathbf b},
\qquad
r_{\mu\lambda}(q)\in\mathbb Z[q,q^{-1}],
\]
where the sum may be infinite.

\item
If $\lambda$ is $\mathbf b$-dominant and $\mathbf b$-typical, then
$\overline{v_\lambda^{\mathbf b}}=v_\lambda^{\mathbf b}$.
\end{enumerate}
\end{proposition}

Brundan proved the following generation result for the monomial basis and
constructed the bar involution independently of Lusztig's general theory.

\begin{theorem}[\cite{brundan2003kazhdan}]
Let $\mathbf b=0^m1^n$. Then every monomial $v_\lambda^{\mathbf b}$ can be
expressed as a possibly infinite linear combination of elements
$u\,v_\mu^{\mathbf b}$, where $\mu$ is $\mathbf b$-typical and
$u\in U_q\mathfrak{sl}_\infty$.
\end{theorem}

Taking the Hecke algebra action into account, this can be reformulated as
follows.

\begin{corollary}
Let $\mathbf b=0^m1^n$. Then every monomial $v_\lambda^{\mathbf b}$ can be
expressed as a possibly infinite linear combination of elements
$u\,v_\mu^{\mathbf b}h$, where $\mu$ is $\mathbf b$-dominant and
$\mathbf b$-typical, $u\in U_q\mathfrak{sl}_\infty$, and
$h\in\mathcal H^{\mathbf b}$.
\end{corollary}

The following extension to arbitrary $\mathbf b$ appears natural.

\begin{conjecture}
Let $\mathbf b\in L(m,n)$. Then every monomial $v_\lambda^{\mathbf b}$ can
be expressed as a possibly infinite linear combination of elements
$u\,v_\mu^{\mathbf b}h$, where $\mu$ is $\mathbf b$-dominant and
$\mathbf b$-typical, $u\in U_q\mathfrak{sl}_\infty$, and
$h\in\mathcal H^{\mathbf b}$.
\end{conjecture}

\begin{remark}
If $\mathbf b=(01)^n$, then every maximal constant interval has length one,
so $\mathcal H^{\mathbf b}\cong\mathbb Q(q)$. Thus the Hecke algebra action
is trivial.
\end{remark}

\begin{theorem}\label{thm:canonical-basis}
\cite{brundan2003kazhdan,cheng2015brundan}
For every $\mathbf b\in L(m,n)$, the completed tensor space
$\widehat{\mathbb T}^{\mathbf b}$ admits a canonical basis
$\{b_\lambda^{\mathbf b}\mid\lambda\in\mathbb Z^{m+n}\}$.
It is uniquely characterized by
$\overline{b_\lambda^{\mathbf b}}=b_\lambda^{\mathbf b}$ and
\[
b_\lambda^{\mathbf b}
=
v_\lambda^{\mathbf b}
+
\sum_{\mu\prec_{\mathbf b}\lambda}
d_{\lambda,\mu}^{\mathbf b}(q)v_\mu^{\mathbf b},
\qquad
d_{\lambda,\mu}^{\mathbf b}(q)\in q\mathbb Z[q].
\]
Equivalently,
$b_\lambda^{\mathbf b}
=\sum_{\mu}d_{\lambda,\mu}^{\mathbf b}(q)v_\mu^{\mathbf b}$, where
$d_{\lambda,\lambda}^{\mathbf b}(q)=1$ and
$d_{\lambda,\mu}^{\mathbf b}(q)=0$ unless
$\mu\preceq_{\mathbf b}\lambda$.
\end{theorem}
```

\subsection{$\mathbf c$-monomial bases}
\label{subsec:c-monomial-bases}

For integers $0\leq r\leq n\leq m$, let $\mathcal C(m,n)_r$ be the set of
words consisting of $m-r$ symbols $0$, $n-r$ symbols $1$, and $r$ symbols
$s$, and set
$\mathcal C(m,n):=\bigsqcup_{r=0}^{n}\mathcal C(m,n)_r$.

We define a partial order $\subseteq$ on $\mathcal C(m,n)$ by declaring
$\mathbf c\subseteq\mathbf c'$ if $\mathbf c$ can be obtained from
$\mathbf c'$ by repeatedly replacing a symbol $s$ by either $01$ or $10$.
Thus replacing an $s$ by $01$ or $10$ moves downward in the partial order.
For example, in $\mathcal C(n,n)$,
\[
0^n1^n\subseteq 0^{n-1}s1^{n-1},
\qquad
(01)^n\subseteq s^n.
\]
Moreover, $\mathcal C(m,n)_0=L(m,n)$ and
$\mathcal C(n,n)_n=\{s^n\}$.

\begin{definition}
Let $\mathbf b,\mathbf b'\in L(m,n)$ be adjacent vertices in the finite
Young lattice. Then there exist words $\mathbf a$ and $\mathbf d$ such that
$\mathbf b=\mathbf a\,01\,\mathbf d$ and
$\mathbf b'=\mathbf a\,10\,\mathbf d$. We define
\[
\mathbf b+\mathbf b':=\mathbf a\,s\,\mathbf d
\in\mathcal C(m,n)_1.
\]
\end{definition}

Consider $\mathbb T^{01}=V\otimes W$ and
$\mathbb T^{10}=W\otimes V$. Let
$U^{01}\subseteq\mathbb T^{01}$ be the subspace spanned by the canonical
basis elements $b^{01}_{(\lambda_1,\lambda_2)}$ belonging to the class under
consideration, and define $U^{10}\subseteq\mathbb T^{10}$ similarly.

\begin{lemma}
There is an isomorphism of $U_q\mathfrak{sl}_\infty$-modules
$\Phi:U^{01}\xrightarrow{\sim}U^{10}$ given by
\[
\Phi\!\left(b^{01}_{(\lambda_1,\lambda_2)}\right)
=
\begin{cases}
b^{10}_{(\lambda_1+1,\lambda_1+1)},
& \lambda_1=\lambda_2,\\[2mm]
b^{10}_{(\lambda_2,\lambda_1)},
& \lambda_1\neq\lambda_2.
\end{cases}
\]
\end{lemma}

In what follows, we identify $U^{01}$ and $U^{10}$ via $\Phi$ and denote the
resulting $U_q\mathfrak{sl}_\infty$-module simply by $U$.

Let $\mathbf c=(c_1,\ldots,c_{m+n-r})\in\mathcal C(m,n)_r$, and choose
$\mathbf b\in L(m,n)$ such that $\mathbf b\subseteq\mathbf c$. Define
\[
T^{\mathbf c}
:=
T^{c_1}\otimes\cdots\otimes T^{c_{m+n-r}},
\qquad
T^{c_i}
:=
\begin{cases}
V, & c_i=0,\\
W, & c_i=1,\\
U, & c_i=s.
\end{cases}
\]
By the preceding lemma, the resulting $U_q\mathfrak{sl}_\infty$-module is
independent, up to the canonical identifications above, of the choice of
$\mathbf b\subseteq\mathbf c$. We regard $T^{\mathbf c}$ as a submodule of
$\mathbb T^{\mathbf b}$.

Now fix $\mathbf b\in L(m,n)$ and
$\mathbf c\in\mathcal C(m,n)_r$ with $\mathbf b\subseteq\mathbf c$.
For each $i$, set
\[
\ell(c_i):=
\begin{cases}
1, & c_i\in\{0,1\},\\
2, & c_i=s,
\end{cases}
\qquad
a_i:=1+\sum_{t<i}\ell(c_t).
\]
Thus the $i$-th factor of $T^{\mathbf c}$ uses the entries
$\lambda_{a_i},\ldots,\lambda_{a_i+\ell(c_i)-1}$ of
$\lambda\in\mathbb Z^{m+n}$.

For $c_i=0$ or $1$, set
\[
v_{\lambda_{a_i}}^{\,b_{a_i};c_i}
:=
\begin{cases}
v_{\lambda_{a_i}}, & c_i=0,\\
w_{\lambda_{a_i}}, & c_i=1.
\end{cases}
\]
If $c_i=s$, then
$(b_{a_i},b_{a_i+1})\in\{(0,1),(1,0)\}$, and we set
\[
v_{\lambda_{a_i},\lambda_{a_i+1}}^{
\,b_{a_i},b_{a_i+1};s}
:=
b_{\lambda_{a_i},\lambda_{a_i+1}}^{
(b_{a_i},b_{a_i+1})}\in U.
\]

We then define
\[
v_\lambda^{\mathbf b\subseteq\mathbf c}
:=
\bigotimes_{i=1}^{m+n-r}
v_{\lambda_{a_i},\ldots,\lambda_{a_i+\ell(c_i)-1}}^{
\,b_{a_i},\ldots,b_{a_i+\ell(c_i)-1};c_i}.
\]
We call $v_\lambda^{\mathbf b\subseteq\mathbf c}$ a
\emph{hypercubic monomial}. The span of all such hypercubic monomials is
naturally identified with $T^{\mathbf c}$.

\begin{theorem}\label{thm:canonical-basis-bc}
Let $\mathbf b\in L(m,n)$ and $\mathbf c\in\mathcal C(m,n)$ satisfy
$\mathbf b\subseteq\mathbf c$. Then
$\widehat T^{\mathbf c}\subseteq\widehat{\mathbb T}^{\mathbf b}$ is stable
under the bar involution and admits a canonical basis
$\{b_\lambda^{\mathbf b\subseteq\mathbf c}
\mid\lambda\in\mathbb Z^{m+n}\}$.

It is uniquely characterized by
$\overline{b_\lambda^{\mathbf b\subseteq\mathbf c}}
=b_\lambda^{\mathbf b\subseteq\mathbf c}$ and
\[
b_\lambda^{\mathbf b\subseteq\mathbf c}
=
v_\lambda^{\mathbf b\subseteq\mathbf c}
+
\sum_{\mu\prec_{\mathbf b}\lambda}
d_{\lambda,\mu}^{\mathbf b\subseteq\mathbf c}(q)
v_\mu^{\mathbf b\subseteq\mathbf c},
\qquad
d_{\lambda,\mu}^{\mathbf b\subseteq\mathbf c}(q)\in q\mathbb Z[q].
\]
Moreover, under the natural inclusion
$\widehat T^{\mathbf c}\subseteq\widehat{\mathbb T}^{\mathbf b}$, one has
$b_\lambda^{\mathbf b\subseteq\mathbf c}=b_\lambda^{\mathbf b}$.
\end{theorem}

\section{Super category $\mathcal O$}
\label{sec:super-category-o}

\subsection{Definitions}
\label{subsec:category-o-definitions}

\subsubsection{Borel subalgebras}
\label{subsubsec:borel-subalgebras}
Let $\mathfrak g=\mathfrak{gl}(m|n)$, and let
\[
\mathfrak h
:=
\bigoplus_{i=1}^{m}\Bbbk e_{ii}
\oplus
\bigoplus_{j=1}^{n}\Bbbk e_{m+j,m+j}
\]
be the standard Cartan subalgebra. We denote by
$\varepsilon_1,\ldots,\varepsilon_m,\delta_1,\ldots,\delta_n\in\mathfrak h^*$
the standard basis, defined by
$\varepsilon_i(e_{kk})=\delta_{i,k}$ and
$\delta_j(e_{m+k,m+k})=\delta_{j,k}$.
We equip $\mathfrak h^*$ with the bilinear form
\[
(\varepsilon_i,\varepsilon_j)=\delta_{i,j},\qquad
(\delta_i,\delta_j)=-\delta_{i,j},\qquad
(\varepsilon_i,\delta_j)=0.
\]

The sets of even and odd roots are
\[
\Delta_{\bar0}
=
\{\varepsilon_i-\varepsilon_j\mid i\neq j\}
\cup
\{\delta_p-\delta_q\mid p\neq q\},
\qquad
\Delta_{\bar1}
=
\{\pm(\varepsilon_i-\delta_j)\mid
1\leq i\leq m,\ 1\leq j\leq n\},
\]
respectively, and $\Delta=\Delta_{\bar0}\sqcup\Delta_{\bar1}$.

We fix the standard system of positive even roots
\[
\Delta_{\bar0}^{\mathrm{st},+}
=
\{\varepsilon_i-\varepsilon_j\mid 1\leq i<j\leq m\}
\cup
\{\delta_p-\delta_q\mid 1\leq p<q\leq n\},
\]
or equivalently the standard Borel subalgebra of
$\mathfrak g_{\bar0}\cong\mathfrak{gl}(m)\oplus\mathfrak{gl}(n)$.

With this even part fixed, the Borel subalgebras of $\mathfrak g$ are
parametrized by $L(m,n)$. Let
$\mathbf b=(b_1,\ldots,b_{m+n})\in L(m,n)$.
Reading $\mathbf b$ from left to right, place
$\varepsilon_1,\ldots,\varepsilon_m$ in the positions occupied by $0$'s and
$\delta_1,\ldots,\delta_n$ in the positions occupied by $1$'s, preserving
the order within each family. This determines a total order
\[
\gamma_1^{\mathbf b}
<
\gamma_2^{\mathbf b}
<
\cdots
<
\gamma_{m+n}^{\mathbf b}
\]
on
$\{\varepsilon_1,\ldots,\varepsilon_m,\delta_1,\ldots,\delta_n\}$.
Set
\[
\Delta^{\mathbf b,+}
:=
\{\gamma_a^{\mathbf b}-\gamma_c^{\mathbf b}
\mid 1\leq a<c\leq m+n\}.
\]

We identify the word $\mathbf b$ with the corresponding Borel subalgebra
\[
\mathbf b
=
\mathfrak h
\oplus
\bigoplus_{\alpha\in\Delta^{\mathbf b,+}}\mathfrak g_\alpha.
\]
Thus, from now on, $L(m,n)$ is identified with the set of Borel
subalgebras of $\mathfrak g$ whose even part is the fixed standard even
Borel, and the same symbol $\mathbf b$ is used for a word and the
corresponding Borel subalgebra.

\begin{definition}\label{RBgdef}
The \emph{odd reflection graph} $OR(\mathfrak g)$ is the edge-colored graph
with vertex set $L(m,n)$ and color set $\Delta_{\bar1}/\{\pm1\}$.
Two vertices $\mathbf b,\mathbf b'\in L(m,n)$ are joined by an edge of
color $\{\pm\alpha\}$ if and only if $\mathbf b$ and $\mathbf b'$ are
related by the odd reflection with respect to the odd simple root $\alpha$.
\end{definition}

The preceding parametrization immediately gives the following.

\begin{proposition}
There is a natural isomorphism of edge-colored graphs
\[
OR(\mathfrak{gl}(m|n))\cong L(m,n).
\]
Under this isomorphism, the color set $\Delta_{\bar1}/\{\pm1\}$ is
naturally identified with the set of boxes in an $m\times n$ rectangle
by assigning $\{\pm(\varepsilon_i-\delta_j)\}$ to the box with coordinates
$(j,m+1-i)$.
\end{proposition}

\begin{definition}\label{ld}
For $\lambda\in\Lambda$, set
\[
A_\lambda
:=
\{\{\pm\alpha\}\in\Delta_{\bar1}/\{\pm1\}
\mid (\lambda,\alpha)\neq0\}.
\]
Using the identification of colors above, we regard $A_\lambda$ as a
subset of the color set of $L(m,n)$ and define
$L(m,n)_\lambda:=L(m,n)/A_\lambda$.
In particular, $L(m,n)_0=L(m,n)$.

For $\mathbf b,\mathbf b'\in L(m,n)$, let
$\overline{\mathbf b}$ and $\overline{\mathbf b'}$ denote their images in
$L(m,n)_\lambda$. We define
$\ell_\lambda(\mathbf b,\mathbf b')$ to be the length of a shortest path
from $\overline{\mathbf b}$ to $\overline{\mathbf b'}$ in
$L(m,n)_\lambda$.
\end{definition}

\subsubsection{$\mathbf b$-Verma modules $M^{\mathbf b}(\lambda)$}
\label{subsubsec:verma-modules}
\begin{definition}
A weight $\lambda\in\mathfrak h^*$ is called \emph{integral} if
$\lambda=\sum_{i=1}^m a_i\varepsilon_i+\sum_{j=1}^n b_j\delta_j$
for some $a_i,b_j\in\mathbb Z$. We denote the set of integral weights by
$\Lambda$.

Let $s\mathcal W$ denote the full subcategory of $\mathfrak g$-supermodules
which are $\mathfrak h$-semisimple, have integral weights, and have
finite-dimensional weight spaces. Thus every $M\in s\mathcal W$ admits a
decomposition
\[
M=\bigoplus_{\lambda\in\Lambda}M_\lambda,
\qquad
M_\lambda
:=
\{v\in M\mid hv=\lambda(h)v\text{ for all }h\in\mathfrak h\},
\]
with $\dim M_\lambda<\infty$ for all $\lambda\in\Lambda$.
\end{definition}

\begin{lemma}[See also {\cite[Lemma~2.2]{brundan2014representations}}]
There exists a map
$\operatorname{par}:\Lambda\to\mathbb Z/2\mathbb Z$ such that the full
subcategory
\[
\mathcal W
:=
\left\{
M\in s\mathcal W
\ \middle|\
M_\lambda\text{ is concentrated in parity }\operatorname{par}(\lambda)
\text{ whenever }M_\lambda\neq0
\right\}
\]
is a Serre subcategory of $s\mathcal W$ containing the trivial module.
Moreover, $s\mathcal W=\mathcal W\oplus\Pi\mathcal W$.
\end{lemma}

Similarly, for $\mathfrak g_{\bar0}$, we define
$s\mathcal W_{\bar0}$ and choose the corresponding subcategory
$\mathcal W_{\bar0}$ compatibly with restriction.

\begin{lemma}[See {\cite[\S6.1]{coulembier2017gorenstein}}]
\label{indres}
The functors
\[
\operatorname{Res}^{\mathfrak g}_{\mathfrak g_{\bar0}}
:
\mathcal W\to\mathcal W_{\bar0},
\qquad
\operatorname{Ind}^{\mathfrak g}_{\mathfrak g_{\bar0}},
\operatorname{Coind}^{\mathfrak g}_{\mathfrak g_{\bar0}}
:
\mathcal W_{\bar0}\to\mathcal W
\]
are exact. Moreover,
\[
\operatorname{Ind}^{\mathfrak g}_{\mathfrak g_{\bar0}}
\dashv
\operatorname{Res}^{\mathfrak g}_{\mathfrak g_{\bar0}}
\dashv
\operatorname{Coind}^{\mathfrak g}_{\mathfrak g_{\bar0}},
\qquad
\operatorname{Ind}^{\mathfrak g}_{\mathfrak g_{\bar0}}
\cong
\operatorname{Coind}^{\mathfrak g}_{\mathfrak g_{\bar0}}.
\]
Furthermore,
\[
\operatorname{Res}^{\mathfrak g}_{\mathfrak g_{\bar0}}
\circ
\operatorname{Ind}^{\mathfrak g}_{\mathfrak g_{\bar0}}
\cong
\Lambda(\mathfrak g_{\bar1})\otimes-.
\]
The unit
\[
\operatorname{Id}_{\mathcal W_{\bar0}}
\longrightarrow
\operatorname{Res}^{\mathfrak g}_{\mathfrak g_{\bar0}}
\circ
\operatorname{Ind}^{\mathfrak g}_{\mathfrak g_{\bar0}}
\]
is a split monomorphism, while the counit
\[
\operatorname{Res}^{\mathfrak g}_{\mathfrak g_{\bar0}}
\circ
\operatorname{Ind}^{\mathfrak g}_{\mathfrak g_{\bar0}}
\longrightarrow
\operatorname{Id}_{\mathcal W_{\bar0}}
\]
is a split epimorphism.
\end{lemma}

\begin{definition}
For $M\in\mathcal W$, define its \emph{character} by
$\operatorname{ch}M:=\sum_{\lambda\in\Lambda}\dim M_\lambda\,e^\lambda$.
\end{definition}

Let $L_{\mathfrak h}(\lambda)$ be the one-dimensional even
$\mathfrak h$-module of weight $\lambda\in\Lambda$.
For $\mathbf b\in L(m,n)$, regard $L_{\mathfrak h}(\lambda)$ as a
$\mathbf b$-module by inflation, with the nilradical of $\mathbf b$
acting trivially. Define the $\mathbf b$-Verma supermodule by
\[
M^{\mathbf b}(\lambda)
:=
\Pi^{\operatorname{par}(\lambda)}
\operatorname{Ind}_{\mathbf b}^{\mathfrak g}
L_{\mathfrak h}(\lambda).
\]
The parity shift is chosen so that $M^{\mathbf b}(\lambda)\in\mathcal W$.

We denote by $v_\lambda^{\mathbf b}$ a highest weight vector of
$M^{\mathbf b}(\lambda)$ and by $L^{\mathbf b}(\lambda)$ its simple top.
For $\mathfrak g_{\bar0}$, the corresponding Verma and simple modules are
denoted by $M_{\bar0}(\lambda)$ and $L_{\bar0}(\lambda)$.

Define the \emph{Berezinian weight} by
$\operatorname{ber}
:=\varepsilon_1+\cdots+\varepsilon_m-(\delta_1+\cdots+\delta_n)$.
A weight $\lambda\in\mathfrak h^*$ is orthogonal to every root if and only if
$\lambda=t\,\operatorname{ber}$ for some $t\in\Bbbk$.

\begin{definition}[Integral Weyl vectors {\cite{brundan2014representations}}]
Define
\[
\rho_{\bar0}
:=
\frac12\sum_{\beta\in\Delta_{\bar0}^+}\beta,
\qquad
\rho_{\bar1}^{\mathbf b}
:=
\frac12\sum_{\gamma\in\Delta_{\bar1}^{\mathbf b,+}}\gamma,
\]
and set $\rho_{\bar1}:=\rho_{\bar1}^{0^m1^n}$.

For the standard Borel, define
\[
\rho
:=
-\varepsilon_2-2\varepsilon_3-\cdots-(m-1)\varepsilon_m
+(m-1)\delta_1+(m-2)\delta_2+\cdots+(m-n)\delta_n.
\]
Equivalently,
$\rho=\rho_{\bar0}-\rho_{\bar1}
-\frac{m+n-1}{2}\operatorname{ber}$.

For $\mathbf b\in L(m,n)$, set
$\rho^{\mathbf b}
:=\rho_{\bar0}-\rho_{\bar1}^{\mathbf b}
-\frac{m+n-1}{2}\operatorname{ber}$.
Then $\rho^{\mathbf b}$ is integral.

If $\alpha$ is a $\mathbf b$-simple odd root and
$r_\alpha\mathbf b$ is obtained from $\mathbf b$ by the corresponding odd
reflection, then
$\rho^{r_\alpha\mathbf b}=\rho^{\mathbf b}+\alpha$.
Moreover, for every $\mathbf b$-simple root $\alpha$,
$(\rho^{\mathbf b},\alpha)=\frac12(\alpha,\alpha)$.
In particular, $(\rho^{\mathbf b},\alpha)=0$ for every isotropic odd
simple root $\alpha$.

It is convenient to encode $\lambda\in\Lambda$ by the $m\mid n$-tuple
$(\lambda_1,\ldots,\lambda_m\mid
\lambda_{m+1},\ldots,\lambda_{m+n})\in\mathbb Z^{m+n}$, where
$\lambda_i=(\lambda+\rho,\varepsilon_i)$ for $1\leq i\leq m$ and
$\lambda_{m+j}=(\lambda+\rho,\delta_j)$ for $1\leq j\leq n$.

We call $\lambda$ \emph{regular} if
$\lambda_i\neq\lambda_j$ whenever
$1\leq i<j\leq m$ or $m+1\leq i<j\leq m+n$.

We call $\lambda$ \emph{antidominant} if
$\lambda_1\leq\cdots\leq\lambda_m$ and
$\lambda_{m+1}\geq\cdots\geq\lambda_{m+n}$, and \emph{dominant} if
$\lambda_1\geq\cdots\geq\lambda_m$ and
$\lambda_{m+1}\leq\cdots\leq\lambda_{m+n}$.

The dot action of
$W\cong\mathfrak S_m\times\mathfrak S_n$ on $\Lambda$ is defined by
\[
w\cdot\lambda
:=
w(\lambda+\rho_{\bar0})-\rho_{\bar0}
=
w(\lambda+\rho)-\rho.
\]

Let $\Lambda^{\mathrm{dom}}$ and $\Lambda^{\mathrm{antidom}}$ denote the
sets of dominant and antidominant weights. For $\lambda\in\Lambda$, let
$\lambda^{\mathrm{dom}}$ and $\lambda^{\mathrm{antidom}}$ denote the unique
dominant and antidominant weights in $W\cdot\lambda$, respectively.
Thus we obtain maps
$()^{\mathrm{dom}}:\Lambda\to\Lambda^{\mathrm{dom}}$ and
$()^{\mathrm{antidom}}:\Lambda\to\Lambda^{\mathrm{antidom}}$.

Let $\Lambda^{\mathrm{reg}}\subseteq\Lambda$ be the set of regular weights
and set
$\Lambda^{\mathrm{reg,dom}}
:=\Lambda^{\mathrm{reg}}\cap\Lambda^{\mathrm{dom}}$.
\end{definition}

\begin{definition}
Let $L=L^{\mathbf b}(\lambda)$ be simple. Its \emph{atypicality} is
\[
\operatorname{aty}L
:=
\max
\left\{
t\ \middle|\
\begin{array}{l}
\text{there exist pairwise orthogonal distinct classes }
\{\pm\alpha_1\},\ldots,\{\pm\alpha_t\}
\in\Delta_{\bar1}/\{\pm1\},\\
(\lambda+\rho^{\mathbf b},\alpha_i)=0
\text{ for }i=1,\ldots,t
\end{array}
\right\}.
\]
This number is independent of the choice of $\mathbf b$. We call $L$
\emph{typical} if $\operatorname{aty}L=0$, and \emph{atypical} otherwise.
\end{definition}

We will use the following standard facts.

\begin{lemma}\label{5.2.ch_eq}
Let $\mathbf b,\mathbf b'\in L(m,n)$ and
$\lambda,\lambda'\in\Lambda$. Then:
\begin{enumerate}
\item
$\operatorname{ch}M^{\mathbf b}(\lambda-\rho^{\mathbf b})
=
\operatorname{ch}M^{\mathbf b'}(\lambda'-\rho^{\mathbf b'})$
if and only if $\lambda=\lambda'$.

\item
$\dim
M^{\mathbf b'}(\lambda-\rho^{\mathbf b'})_
{\lambda-\rho^{\mathbf b}}=1$.

\item
\[
\dim\operatorname{Hom}_{\mathfrak g}
\bigl(
M^{\mathbf b}(\lambda-\rho^{\mathbf b}),
M^{\mathbf b'}(\lambda-\rho^{\mathbf b'})
\bigr)
=1.
\]

\item
Suppose that $\mathbf b'$ is obtained from $\mathbf b$ by the odd
reflection with respect to a $\mathbf b$-simple odd root $\alpha$.
\begin{enumerate}
\item
If $(\lambda,\alpha)\neq0$, then
\[
M^{\mathbf b}(\lambda-\rho^{\mathbf b})
\cong
M^{\mathbf b'}(\lambda-\rho^{\mathbf b'}),
\qquad
L^{\mathbf b}(\lambda-\rho^{\mathbf b})
\cong
L^{\mathbf b'}(\lambda-\rho^{\mathbf b'}).
\]

\item
If $(\lambda,\alpha)=0$, then
\[
M^{\mathbf b}(\lambda-\rho^{\mathbf b})
\not\cong
M^{\mathbf b'}(\lambda-\rho^{\mathbf b'}),
\]
whereas
\[
L^{\mathbf b}(\lambda-\rho^{\mathbf b})
\cong
L^{\mathbf b'}(\lambda-\rho^{\mathbf b})
=
L^{\mathbf b'}
(\lambda+\alpha-\rho^{\mathbf b'}).
\]
\end{enumerate}
\end{enumerate}
\end{lemma}

\subsubsection{Category $\mathcal O$}
\label{subsubsec:category-o}
\begin{definition}
Fix $\mathbf b\in L(m,n)$. We define $\mathcal O$ to be the Serre
subcategory of $\mathcal W$ generated by
$\{L^{\mathbf b}(\lambda)\mid\lambda\in\Lambda\}$.
By \cref{5.2.ch_eq}, as a Serre subcategory of $\mathcal W$, the category
$\mathcal O$ is independent of the choice of $\mathbf b$, although its
highest weight structure depends on $\mathbf b$.
\end{definition}

It is well known that $\mathcal O$ contains
$M^{\mathbf b}(\lambda)$ for every $\mathbf b\in L(m,n)$ and
$\lambda\in\Lambda$.

Similarly, replacing $\mathfrak g$ by $\mathfrak g_{\bar0}$, we define
the BGG category $\mathcal O_{\bar0}$ as a full subcategory of
$\mathcal W_{\bar0}$.

The functors
$\operatorname{Res}^{\mathfrak g}_{\mathfrak g_{\bar0}}$,
$\operatorname{Ind}^{\mathfrak g}_{\mathfrak g_{\bar0}}$, and
$\operatorname{Coind}^{\mathfrak g}_{\mathfrak g_{\bar0}}$
are exact and restrict to functors between $\mathcal O_{\bar0}$ and
$\mathcal O$.

For each $\mathbf b\in L(m,n)$, the category $\mathcal O$ carries a
highest weight structure whose standard objects are the Verma modules
$M^{\mathbf b}(\lambda)$. The corresponding weight poset is
$(\Lambda,\leq_{\mathbf b})$, where
$\mu\leq_{\mathbf b}\lambda$ if and only if $\lambda-\mu$ is a
nonnegative integral combination of roots in $\Delta^{\mathbf b,+}$.

For $\lambda\in\Lambda$, let $P^{\mathbf b}(\lambda)$ and
$I^{\mathbf b}(\lambda)$ denote the projective cover and injective hull
of $L^{\mathbf b}(\lambda)$ in $\mathcal O$, respectively. Similarly,
let $P_{\bar0}(\lambda)$ and $I_{\bar0}(\lambda)$ denote the projective
cover and injective hull of $L_{\bar0}(\lambda)$ in
$\mathcal O_{\bar0}$.

Fix $\mathbf b\in L(m,n)$. An object $X\in\mathcal O$ is said to have an
$M^{\mathbf b}$-\emph{flag} if it admits a finite filtration whose
subquotients are isomorphic to modules of the form
$M^{\mathbf b}(\lambda)$. We denote by
$\mathcal F M^{\mathbf b}\subseteq\mathcal O$ the full subcategory of
such objects.

We denote by $\mathcal F M_{\bar0}\subseteq\mathcal O$ the full
subcategory consisting of objects $X$ such that
$\operatorname{Res}^{\mathfrak g}_{\mathfrak g_{\bar0}}X$ admits a
Verma flag in $\mathcal O_{\bar0}$.

For $X\in\mathcal F M^{\mathbf b}$, let
$(X:M^{\mathbf b}(\lambda))$ denote the multiplicity of
$M^{\mathbf b}(\lambda)$ in an $M^{\mathbf b}$-flag of $X$.
This multiplicity is independent of the choice of flag.

\begin{proposition}\label{O.indproj}
Fix $\mathbf b_0\in L(m,n)$ and $\lambda\in\Lambda$. Then:
\begin{enumerate}
\item
The category $\mathcal F M^{\mathbf b_0}$ is closed under direct
summands.

\item
We have
\[
P^{\mathbf b_0}(\lambda)
\in
\bigcap_{\mathbf b\in L(m,n)}
\mathcal F M^{\mathbf b}.
\]

\item
We have
$\mathcal F M^{\mathbf b_0}\subseteq\mathcal F M_{\bar0}$.

\item
For every $\mathbf b\in L(m,n)$, BGG reciprocity holds:
\[
\bigl(P^{\mathbf b}(\lambda):M^{\mathbf b}(\mu)\bigr)
=
\bigl[M^{\mathbf b}(\mu):L^{\mathbf b}(\lambda)\bigr].
\]
\end{enumerate}
\end{proposition}

\begin{proof}
We only prove~(2). For every $\mu\in\Lambda$,
$\operatorname{Ind}_{\mathfrak g_{\bar0}}^{\mathfrak g}
M_{\bar0}(\mu)$ belongs to
$\bigcap_{\mathbf b\in L(m,n)}\mathcal F M^{\mathbf b}$.
Since $P_{\bar0}(\lambda)$ admits an even Verma flag, it follows that
$\operatorname{Ind}_{\mathfrak g_{\bar0}}^{\mathfrak g}
P_{\bar0}(\lambda)$ belongs to the same intersection.

Moreover, $P^{\mathbf b_0}(\lambda)$ is a direct summand of
$\operatorname{Ind}_{\mathfrak g_{\bar0}}^{\mathfrak g}
P_{\bar0}(\lambda)$. Hence~(2) follows from~(1).
\end{proof}

We will also use the following consequence of Irving's result.

\begin{proposition}[\cite{mazorchuk2014parabolic}]
\label{antidom}
For $\lambda\in\Lambda$, the following are equivalent:
\begin{enumerate}
\item
$\lambda$ is antidominant.

\item
There exist $M\in\mathcal F M_{\bar0}$ and an injective homomorphism
$L^{0^m1^n}(\lambda)\hookrightarrow M$.
\end{enumerate}
\end{proposition}

\begin{proof}
We have an embedding
\[
L_{\bar0}(\lambda)
\hookrightarrow
\operatorname{Res}^{\mathfrak g}_{\mathfrak g_{\bar0}}
L^{0^m1^n}(\lambda).
\]
The assertion follows from Irving's criterion together with the
induction--restriction adjunction.
\end{proof}

\subsubsection{Khovanov arc diagrams}
\label{subsubsec:khovanov-arc-diagrams}
For the material in this subsection, we refer the reader to
\cite{brundanstroppelI,brundanstroppelII,brundanstroppelIV,
heidersdorf2017mixed}.

Recall that $0^m1^n$ and $1^n0^m$ correspond to the distinguished and
anti-distinguished Borel subalgebras, respectively. Set
\[
\mathfrak g_1:=(0^m1^n)_{\bar1},
\qquad
\mathfrak g_{-1}:=(1^n0^m)_{\bar1},
\]
and
$\mathfrak g_{\geq0}:=\mathfrak g_{\bar0}\oplus\mathfrak g_1$ and
$\mathfrak g_{\leq0}:=\mathfrak g_{-1}\oplus\mathfrak g_{\bar0}$.
Then
$\mathfrak g=\mathfrak g_{-1}\oplus\mathfrak g_{\bar0}\oplus\mathfrak g_1$
is the natural $\mathbb Z$-grading of $\mathfrak{gl}(m|n)$, and
$[\mathfrak g_1,\mathfrak g_1]
=
[\mathfrak g_{-1},\mathfrak g_{-1}]
=0$.

Define the Kac and opposite Kac functors by
\[
K_{\geq0}
:=
\operatorname{Ind}_{\mathfrak g_{\geq0}}^{\mathfrak g}
\circ
\operatorname{Infl}_{\mathfrak g_{\bar0}}^{\mathfrak g_{\geq0}},
\qquad
K_{\leq0}
:=
\operatorname{Ind}_{\mathfrak g_{\leq0}}^{\mathfrak g}
\circ
\operatorname{Infl}_{\mathfrak g_{\bar0}}^{\mathfrak g_{\leq0}}.
\]
Here $\mathfrak g_1$, respectively $\mathfrak g_{-1}$, acts trivially
under inflation. Both functors are exact.

\begin{definition}
For $\sigma\in\{\geq0,\leq0\}$, let
$\mathcal F K_\sigma\subseteq\mathcal O$ be the full subcategory of
objects admitting a finite filtration whose subquotients are of the form
$K_\sigma(L_{\bar0}(\lambda))$ for $\lambda\in\Lambda$.
\end{definition}

\begin{definition}
Let $s\mathcal F$ be the full subcategory of $s\mathcal W$ consisting of
finite-dimensional modules, and set
$\mathcal F:=s\mathcal F\cap\mathcal W$. Thus
$s\mathcal F=\mathcal F\oplus\Pi\mathcal F$.
Similarly, let $\mathcal F_{\bar0}\subseteq s\mathcal F_{\bar0}$ be
defined using the parity convention compatible with
$\operatorname{Res}^{\mathfrak g}_{\mathfrak g_{\bar0}}$.

For $\lambda\in\Lambda^{\mathrm{reg,dom}}$, set
\[
K(\lambda):=K_{\geq0}(L_{\bar0}(\lambda)),
\qquad
L(\lambda):=L^{0^m1^n}(\lambda).
\]
With the modules $K(\lambda)$ as standard objects, $\mathcal F$ is a
highest weight category.
\end{definition}

A \emph{number line} is a linearly ordered set identified with $\mathbb Z$,
drawn from left to right in increasing order. A \emph{weight diagram} is
a labelling of the number line by
$\circ,\times,\vee,\wedge$, with all but finitely many vertices labelled
by $\wedge$.

The Bruhat order on weight diagrams is generated by interchanging a
$\vee$ and a $\wedge$, with moving a $\vee$ to the right increasing the
weight. We write $\lambda\sim\mu$ if $\mu$ can be obtained from $\lambda$
by permuting only the $\vee$'s and $\wedge$'s, while keeping the positions
of the $\circ$'s and $\times$'s fixed. Equivalently,
$\lambda\sim\mu$ if and only if there exists a weight diagram $\nu$ such
that $\lambda\leq\nu\geq\mu$. The equivalence classes are called
\emph{blocks}.

The \emph{defect} $\operatorname{def}(\lambda)$ of a weight diagram
$\lambda$ is the number of $\vee$'s. It is constant on each block and
hence defines the defect of a block.

A \emph{cup diagram} is obtained by attaching pairwise non-intersecting
cups and rays to some vertices of the number line, with cups drawn below
the number line and rays extending downward to infinity. Vertices not
incident with a cup or ray are called \emph{free vertices}. Cap diagrams
are defined similarly, with caps drawn above the number line and rays
extending upward to infinity.

If $c$ is a cup diagram and $\lambda$ is a weight diagram on the same
number line, we write $c\lambda$ for the diagram obtained by placing $c$
below $\lambda$. It is called \emph{oriented} if every free vertex is
labelled by $\circ$ or $\times$, the endpoints of each cup are labelled
by one $\vee$ and one $\wedge$, and every vertex at the top of a ray is
labelled by $\wedge$. Since a weight diagram contains only finitely many
$\vee$'s, every oriented cup diagram contains only finitely many cups.

For a cap diagram $c$, we define $\lambda c$ similarly by placing $c$
above $\lambda$. Let $c^*$ denote the reflection of $c$ in the number
line. We call $\lambda c$ oriented if $c^*\lambda$ is an oriented cup
diagram. An oriented cup, cap, or circle is called
\emph{anti-clockwise} if its left endpoint is labelled by $\vee$ and its
right endpoint by $\wedge$, and \emph{clockwise} otherwise.

\begin{proposition}[\cite{brundanstroppelI}]
For every weight diagram $\lambda$, there exists a unique cup diagram
$\underline{\lambda}$ such that $\underline{\lambda}\lambda$ is oriented
and all its cups are anti-clockwise. Equivalently, weight diagrams are in
bijection with pairs $(\underline{\lambda},\lambda_{\circ,\times})$,
where $\lambda_{\circ,\times}$ records the positions of the $\circ$'s
and $\times$'s.
\end{proposition}

Let $\overline{\lambda}:=(\underline{\lambda})^*$ be the corresponding
cap diagram. Thus $\lambda\overline{\lambda}$ is oriented and all its
caps are anti-clockwise. For weight diagrams $\lambda$ and $\mu$, the
conditions that $\underline{\lambda}\mu$ is oriented and that
$\lambda\overline{\mu}$ is oriented are equivalent. In this case, we
write $\lambda\subset\mu$.

\begin{lemma}[{\cite[Lemma~2.4]{brundanstroppelI}}]
For every weight diagram $\lambda$,
\[
\#\{\mu\mid\lambda\subset\mu\}
=
2^{\operatorname{def}(\lambda)},
\qquad
\#\{\mu\mid\mu\subset\lambda\}<\infty.
\]
\end{lemma}

Let $\Gamma$ be the set of all weight diagrams on $\mathbb Z$ in which
exactly $m$ vertices are labelled by $\times$ or $\vee$, exactly $n$
vertices are labelled by $\circ$ or $\vee$, and all remaining vertices
are labelled by $\wedge$.

We identify $\Lambda^{\mathrm{reg,dom}}$ with $\Gamma$ as follows. For
$\lambda\in\Lambda^{\mathrm{reg,dom}}$, set
\[
I_\times(\lambda)
:=
\{\lambda_1,\lambda_2-1,\ldots,\lambda_m-m+1\},
\]
and
\[
I_\circ(\lambda)
:=
\{1-m-\lambda_{m+1},2-m-\lambda_{m+2},
\ldots,n-m-\lambda_{m+n}\}.
\]
The elements of $I_\times(\lambda)\cap I_\circ(\lambda)$ are labelled by
$\vee$; the remaining elements of $I_\times(\lambda)$ by $\times$, the
remaining elements of $I_\circ(\lambda)$ by $\circ$, and all other
integers by $\wedge$. This determines the corresponding weight diagram
uniquely.

The number of $\vee$'s is the degree of atypicality of $\lambda$. If this
number is $r$, then $0\leq r\leq\min(m,n)$, and $\lambda$ is called
\emph{maximally atypical} if $r=\min(m,n)$. Thus the degree of atypicality
agrees with the defect of the corresponding weight diagram. Under the
identification $\Lambda^{\mathrm{reg,dom}}\cong\Gamma$, the Bruhat order
agrees with the diagrammatic Bruhat order.

The degree of an oriented cup diagram $c\lambda$ is the number of its
clockwise cups and is denoted by $\deg(c\lambda)$. Similarly, the degree
of an oriented cap diagram $\lambda c$ is the number of its clockwise
caps and is denoted by $\deg(\lambda c)$.

\begin{theorem}
[\cite{brundanstroppelI,brundanstroppelII,brundanstroppelIV}]
\label{Fhw}
The category $\mathcal F$ admits a Koszul grading. For
$\lambda\in\Lambda^{\mathrm{reg,dom}}$, the projective module
$P(\lambda)$ has a multiplicity-free $K$-flag whose subquotients are
precisely
\[
K(\mu)\langle\deg(\mu\overline{\lambda})\rangle
\qquad
(\lambda\subset\mu).
\]
Moreover, the composition factors of $K(\lambda)$ are precisely
\[
L(\mu)\langle\deg(\underline{\mu}\lambda)\rangle
\qquad
(\mu\subset\lambda),
\]
each with multiplicity one.
\end{theorem}

Brundan and Stroppel proved considerably stronger results: they described
the basic algebra of each block and the translation functors on
$\mathcal F$ in terms of arc diagrams.

\subsection{Categorification of $\mathbf c$-monomial bases}
\label{subsec:categorification-c-monomial}

\subsubsection{Refined BGG reciprocity}
\label{subsubsec:refined-bgg-reciprocity}

Let $\mathbf b\in L(m,n)$ and $\mathbf c\in\mathcal C(m,n)$ satisfy
$\mathbf b\subseteq\mathbf c$. We naturally identify $\mathbf c$ with the
corresponding parabolic subalgebra. Let $\mathfrak l_{\mathbf c}$ denote its
Levi subalgebra. Thus $\mathfrak l_{\mathbf c}$ is generated by
$\mathfrak h$ together with the mutually orthogonal copies of
$\mathfrak{gl}(1|1)$ corresponding to the symbols $s$ in $\mathbf c$.

For $\lambda\in\mathfrak h^*$, let
$P_{\mathfrak l_{\mathbf c}}^{\mathfrak l_{\mathbf c}\cap\mathfrak b}(\lambda)$,
$M_{\mathfrak l_{\mathbf c}}^{\mathfrak l_{\mathbf c}\cap\mathfrak b}(\lambda)$,
and
$L_{\mathfrak l_{\mathbf c}}^{\mathfrak l_{\mathbf c}\cap\mathfrak b}(\lambda)$
denote the projective cover, Verma module, and simple module for
$\mathfrak l_{\mathbf c}$, respectively.

We define the $\mathbf b\subseteq\mathbf c$-parabolic projective module by
\[
U^{\mathbf b\subseteq\mathbf c}(\lambda)
:=
\operatorname{Ind}_{\mathbf c}^{\mathfrak g}
\operatorname{Infl}_{\mathfrak l_{\mathbf c}}^{\mathbf c}
P_{\mathfrak l_{\mathbf c}}^{\mathfrak l_{\mathbf c}\cap\mathfrak b}(\lambda),
\]
and similarly
\[
N^{\mathbf b\subseteq\mathbf c}(\lambda)
:=
\operatorname{Ind}_{\mathbf c}^{\mathfrak g}
\operatorname{Infl}_{\mathfrak l_{\mathbf c}}^{\mathbf c}
L_{\mathfrak l_{\mathbf c}}^{\mathfrak l_{\mathbf c}\cap\mathfrak b}(\lambda).
\]
The $\mathfrak b$-Verma module can likewise be written as
\[
M^{\mathfrak b}(\lambda)
\cong
\operatorname{Ind}_{\mathbf c}^{\mathfrak g}
\operatorname{Infl}_{\mathfrak l_{\mathbf c}}^{\mathbf c}
M_{\mathfrak l_{\mathbf c}}^{\mathfrak l_{\mathbf c}\cap\mathfrak b}(\lambda).
\]

\begin{proposition}\cite{hirota2025odd}
\label{prop:ringel-subcategory-bc}
Let $B$ be a block of category $\mathcal O$ for $\mathfrak l_{\mathbf c}$.
Then
\[
\left\{
N^{\mathbf b\subseteq\mathbf c}(\lambda)
\ \middle|\
L_{\mathfrak l_{\mathbf c}}^{\mathfrak l_{\mathbf c}\cap\mathfrak b}(\lambda)
\in B
\right\}
\]
is a semibrick in $\mathcal O$. Hence, by Ringel's classical result, its
extension closure
\[
\operatorname{Filt}_{\mathcal O}
\left\{
N^{\mathbf b\subseteq\mathbf c}(\lambda)
\ \middle|\
L_{\mathfrak l_{\mathbf c}}^{\mathfrak l_{\mathbf c}\cap\mathfrak b}(\lambda)
\in B
\right\}
\]
is an extension-closed abelian subcategory of $\mathcal O$.

Moreover,
$\operatorname{Ind}_{\mathbf c}^{\mathfrak g}
\operatorname{Infl}_{\mathfrak l_{\mathbf c}}^{\mathbf c}$
induces an equivalence between $B$ and this subcategory.
\end{proposition}

\begin{proposition}\label{prop:N-as-image}
Let $\mathbf c\in\mathcal C(m,n)$. Denote by
$0^{\mathbf c}\in L(m,n)$ the element obtained from $\mathbf c$ by replacing
every symbol $s$ by $01$, and by $1^{\mathbf c}\in L(m,n)$ the element
obtained by replacing every $s$ by $10$. Then, for every $\lambda$,
\[
N^{0^{\mathbf c}\subseteq\mathbf c}(\lambda)
\cong
\operatorname{Im}\left(
M^{0^{\mathbf c}}(\lambda)
\longrightarrow
M^{1^{\mathbf c}}
\bigl(\lambda+\rho^{0^{\mathbf c}}-\rho^{1^{\mathbf c}}\bigr)
\right).
\]
\end{proposition}

In what follows, we abbreviate
$U^{\mathbf c}(\lambda)
:=U^{0^{\mathbf c}\subseteq\mathbf c}(\lambda)$ and
$N^{\mathbf c}(\lambda)
:=N^{0^{\mathbf c}\subseteq\mathbf c}(\lambda)$.

For fixed $\mathbf c\in\mathcal C(m,n)$, an object $X\in\mathcal O$ is said
to have a $U^{\mathbf c}$-flag, respectively an $N^{\mathbf c}$-flag, if it
admits a finite filtration whose subquotients are isomorphic to modules
$U^{\mathbf c}(\lambda)$, respectively $N^{\mathbf c}(\lambda)$.
We denote the corresponding full subcategories by
$\mathcal F U^{\mathbf c}$ and $\mathcal F N^{\mathbf c}$.

For such $X$, we write $(X:U^{\mathbf c}(\lambda))$ and
$(X:N^{\mathbf c}(\lambda))$ for the corresponding flag multiplicities.
These are independent of the choice of filtration, since the characters of
the modules $U^{\mathbf c}(\lambda)$, respectively $N^{\mathbf c}(\lambda)$,
are linearly independent.

Set
\[
\Delta^{\mathbf c}
:=
\left(
\bigcup_{\mathbf b\subseteq\mathbf c}\Delta^{\mathbf b,+}
\right)
\setminus
\left(
\bigcap_{\mathbf b\subseteq\mathbf c}\Delta^{\mathbf b,+}
\right).
\]
Thus $\Delta^{\mathbf c}$ consists precisely of the roots whose signs vary
among the Borel subalgebras $\mathbf b\subseteq\mathbf c$.

\begin{definition}
We recall the Duflo--Serganova functor; see \cite{gorelik2022duflo}.
Set
$X:=\{x\in\mathfrak g_{\bar1}\mid[x,x]=0\}$.
For $x\in X$ and $M\in\mathfrak g\text{-sMod}$, define
\[
\operatorname{DS}_xM
:=
\ker x_M/\operatorname{Im}x_M,
\]
where $x_M$ denotes the action of $x$ on $M$.
\end{definition}

\begin{proposition}[\cite{gorelik2022duflo}]
\label{prop:DS_monoidal}
For $x\in X$, the superspace
$\mathfrak g_x:=\operatorname{DS}_x\mathfrak g$ naturally inherits the
structure of a Lie superalgebra. Moreover,
\[
\operatorname{DS}_x:
\mathfrak g\text{-sMod}
\longrightarrow
\mathfrak g_x\text{-sMod}
\]
is naturally a symmetric monoidal $\Bbbk$-linear functor.
\end{proposition}

\begin{definition}\label{def:DS_alpha}
For an odd root $\alpha$, choose a nonzero root vector
$e_\alpha\in\mathfrak g_\alpha$ and set
$\operatorname{DS}_\alpha:=\operatorname{DS}_{e_\alpha}$.
\end{definition}

\begin{theorem}\label{thm:DS-characterization-FU}
Let $\mathbf b\in L(m,n)$ and $\mathbf c\in\mathcal C(m,n)$ satisfy
$\mathbf b\subseteq\mathbf c$. Then
\[
\mathcal F U^{\mathbf c}
=
\left\{
M\in\mathcal F N^{\mathbf c}
\ \middle|\
\operatorname{DS}_{\alpha}M=0
\text{ for all }\alpha\in\Delta^{\mathbf c}
\right\}.
\]
Moreover,
\[
\mathcal F M^{\mathbf b}
=
\left\{
M\in\mathcal F N^{\mathbf c}
\ \middle|\
\operatorname{DS}_{\alpha}M=0
\text{ for all }
\alpha\in\Delta^{\mathbf b,+}\cap\Delta^{\mathbf c}
\right\}.
\]
\end{theorem}

\begin{proof}
Let $\mathfrak u_{\mathbf c}^{-}$ be the opposite nilradical associated with
$\mathbf c$. For an $\mathfrak l_{\mathbf c}$-module $M$, PBW gives
\[
\operatorname{Res}_{\mathfrak l_{\mathbf c}}^{\mathfrak g}
\operatorname{Ind}_{\mathbf c}^{\mathfrak g}
\operatorname{Infl}_{\mathfrak l_{\mathbf c}}^{\mathbf c}M
\cong
U(\mathfrak u_{\mathbf c}^{-})\otimes M.
\]
In the present setting, this identifies with
$S(\mathfrak u_{\mathbf c}^{-})\otimes M$ as
$\mathfrak l_{\mathbf c}$-modules.

Let $x\in\mathfrak l_{\mathbf c}$ be odd with $x^2=0$. By the monoidal
property of the Duflo--Serganova functor,
\[
\operatorname{DS}_x
\left(
\operatorname{Ind}_{\mathbf c}^{\mathfrak g}
\operatorname{Infl}_{\mathfrak l_{\mathbf c}}^{\mathbf c}M
\right)
\cong
S(\operatorname{DS}_x\mathfrak u_{\mathbf c}^{-})
\otimes
\operatorname{DS}_xM.
\]
Hence the left-hand side vanishes if and only if
$\operatorname{DS}_xM=0$. Applying this observation to the root vectors
corresponding to $\alpha\in\Delta^{\mathbf c}$, respectively
$\alpha\in\Delta^{\mathbf b,+}\cap\Delta^{\mathbf c}$, gives the two
characterizations.
\end{proof}

\begin{corollary}
For $\mathbf c\in\mathcal C(m,n)$,
\[
\mathcal F U^{\mathbf c}
=
\bigcap_{\mathbf b\subseteq\mathbf c}
\mathcal F M^{\mathbf b}.
\]
In particular, every projective module belongs to
$\mathcal F U^{\mathbf c}$.
\end{corollary}

\begin{theorem}\label{thm:U-N-reciprocity}
Let $\mathbf b\in L(m,n)$ and $\mathbf c\in\mathcal C(m,n)$ satisfy
$\mathbf b\subseteq\mathbf c$. Then, for all $\lambda,\mu$,
\[
\bigl(P^{\mathbf b}(\lambda):
U^{\mathbf b\subseteq\mathbf c}(\mu)\bigr)
=
\bigl[
N^{\mathbf b\subseteq\mathbf c}(\mu):
L^{\mathbf b}(\lambda)
\bigr].
\]
\end{theorem}

\begin{proof}
Combining BGG reciprocity for $\mathfrak g$ with BGG reciprocity for
$\mathfrak l_{\mathbf c}$ gives
\[
\begin{aligned}
\sum_{\mu}
\bigl(P^{\mathbf b}(\lambda):U^{\mathbf b\subseteq\mathbf c}(\mu)\bigr)
\bigl(U^{\mathbf b\subseteq\mathbf c}(\mu):M^{\mathbf b}(\nu)\bigr)
&=
\bigl[M^{\mathbf b}(\nu):L^{\mathbf b}(\lambda)\bigr] \\
&=
\sum_{\mu}
\bigl(U^{\mathbf b\subseteq\mathbf c}(\mu):M^{\mathbf b}(\nu)\bigr)
\bigl[N^{\mathbf b\subseteq\mathbf c}(\mu):L^{\mathbf b}(\lambda)\bigr].
\end{aligned}
\]

To justify coefficient comparison in the infinite highest weight category,
we pass to the  truncations of Brundan--Losev--Webster \cite{brundan2017tensor}. For a suitable
finite interval $J$, let
$\pi_J:\mathcal O\to\mathcal O_J$ denote the corresponding Serre quotient.
The category $\mathcal O_J$ is again highest weight, and for $J$ sufficiently
large the relevant multiplicities are unchanged. Hence the usual finite
highest weight argument yields
\[
\bigl(P^{\mathbf b}(\lambda):
U^{\mathbf b\subseteq\mathbf c}(\mu)\bigr)
=
\bigl[
N^{\mathbf b\subseteq\mathbf c}(\mu):
L^{\mathbf b}(\lambda)
\bigr].
\]
\end{proof}

\subsubsection{Koszul grading}
\label{subsubsec:koszul-grading}
The following theorem extends results of Cheng--Lam--Wang\cite{cheng2015brundan}, who proved the
ungraded case $q=1$ for
$\mathbf c\in L(m,n)\cup\mathcal C(m,n)_1$, and of
Brundan--Losev--Webster\cite{brundan2017tensor}, who proved the graded case for
$\mathbf c\in L(m,n)$.

\begin{theorem}\label{thm:standard-koszul-kl-bc}
The following statements hold.
\begin{enumerate}
\item
The category $\mathcal O$ admits a Koszul graded lift.

\item
For every $\mathbf b\in L(m,n)$ and
$\mathbf c\in\mathcal C(m,n)$ with $\mathbf b\subset\mathbf c$, the modules
$P^{\mathbf b}(\mu)$,
$U^{\mathbf b\subset\mathbf c}(\lambda)$,
$N^{\mathbf b\subset\mathbf c}(\lambda)$, and
$L^{\mathbf b}(\mu)$ admit compatible graded lifts.

\item
The Grothendieck group $K_0(\mathcal F M^{\mathbf b})$ carries a natural
$U_q\mathfrak{sl}_\infty$-module structure. Under the isomorphism
\[
K_0(\mathcal F M^{\mathbf b})
\cong
\mathbb T^{\mathbf b},
\]
we have
\[
[U^{\mathbf b\subset\mathbf c}
(\lambda-\rho^{\mathbf b})\langle n\rangle]
\longmapsto
q^n v_\lambda^{\mathbf b\subset\mathbf c},
\qquad
[P^{\mathbf b}(\lambda-\rho^{\mathbf b})\langle n\rangle]
\longmapsto
q^n b_\lambda^{\mathbf b}.
\]
Consequently,
\[
\begin{aligned}
d_{\mu,\lambda}^{\mathbf b\subset\mathbf c}(q)
&=
\sum_{r\in\mathbb Z}
\bigl(
P^{\mathbf b}(\mu-\rho^{\mathbf b}):
U^{\mathbf b\subset\mathbf c}
(\lambda-\rho^{\mathbf b})\langle r\rangle
\bigr)q^r \\
&=
\sum_{r\in\mathbb Z}
\bigl[
N^{\mathbf b\subset\mathbf c}
(\lambda-\rho^{\mathbf b}):
L^{\mathbf b}(\mu-\rho^{\mathbf b})\langle r\rangle
\bigr]q^r.
\end{aligned}
\]
\end{enumerate}
\end{theorem}

\begin{proof}
Part~(1) is due to Brundan--Losev--Webster.
Parts~(2) and~(3) follow from the case $\mathbf c=\mathbf b$ established
by Brundan--Losev--Webster, together with the graded
$\mathfrak{gl}(1|1)$ analysis governing the passage from
$\mathbf b$ to $\mathbf c$.
\end{proof}

\begin{remark}
In the special case $\mathbf c=\mathbf b$, the equality
\[
\bigl(P^{\mathbf b}(\mu):M^{\mathbf b}(\lambda)\bigr)
=
d_{\mu,\lambda}^{\mathbf b}(1)
=
\bigl[M^{\mathbf b}(\lambda):L^{\mathbf b}(\mu)\bigr]
\]
is the super Kazhdan--Lusztig conjecture. It was conjectured by Brundan
for $\mathbf b=0^m1^n$ and first proved by Cheng--Lam--Wang using
super duality.
\end{remark}

\begin{remark}
Let $\mathbf b\in L(m,n)$ and $\mathbf c\in\mathcal C(m,n)$ satisfy
$\mathbf b\subset\mathbf c$. For a graded lift of
$N^{\mathbf b\subset\mathbf c}(\lambda)$, the grading filtration agrees
with the radical filtration. Hence
$d_{\mu,\lambda}^{\mathbf b\subset\mathbf c}(q)$ records the radical
layers of $N^{\mathbf b\subset\mathbf c}(\lambda)$:
\[
d_{\mu,\lambda}^{\mathbf b\subset\mathbf c}(q)
=
\sum_{r\geq0}
\left[
\frac{\operatorname{rad}^r
N^{\mathbf b\subset\mathbf c}(\lambda)}
{\operatorname{rad}^{r+1}
N^{\mathbf b\subset\mathbf c}(\lambda)}
:
L^{\mathbf b}(\mu)
\right]q^r.
\]
In particular,
$d_{\mu,\lambda}^{\mathbf b\subset\mathbf c}(1)
=
[N^{\mathbf b\subset\mathbf c}(\lambda):L^{\mathbf b}(\mu)]$.
\end{remark}

\begin{example}
\begin{enumerate}
\item
Chen--Coulembier--Mazorchuk\cite{chen2021translated} showed that
$M^{0^m1^n}(\lambda)$ has simple socle. Hence
$M^{0^m1^n}(\lambda)$ is rigid.

\item
We will see in \cref{thm:socle_simple} that
$N^{s^n}(\lambda)$ is rigid under a suitable condition.
\end{enumerate}
\end{example}

\subsubsection{Closed formula for odd reflections}
\label{subsubsec:closed-form-odd-reflections}

\begin{theorem}
For \(\mathbf b,\mathbf b'\in L(m,n)\) and \(\lambda\in\Lambda\), one has
\[
  \sum_{r\geq 0}
  \bigl[
    M^{\mathbf b}(\lambda-\rho^{\mathbf b})
    :
    L^{\mathbf b'}(\lambda-\rho^{\mathbf b'})\langle -r\rangle
  \bigr]q^r
  =
  q^{\ell_{\lambda}(\mathbf b,\mathbf b')},
\]
where \(\ell_{\lambda}(\mathbf b,\mathbf b')\) is the combinatorial quantity
defined in \cref{ld}.
\end{theorem}
\begin{proof}
We argue by induction on \(\ell_{0}(\mathbf b,\mathbf b')\).  If
\(\ell_{\lambda}(\mathbf b,\mathbf b')=0\), equivalently
\(\mathbf b=\mathbf b'\) in \(L(m,n)_\lambda\), the claim is clear.
Suppose that \(\mathbf b'\) and \(\mathbf b''\) are adjacent by the odd
reflection corresponding to an odd root \(\alpha\), and that
\[
  \ell_{0}(\mathbf b,\mathbf b')+1
  =
  \ell_{0}(\mathbf b,\mathbf b'').
\]
Assume the statement for \(\mathbf b'\).  We prove it for \(\mathbf b''\).
If \((\lambda,\alpha)\neq 0\), then the odd reflection gives
\[
  M^{\mathbf b'}(\lambda-\rho^{\mathbf b'})
  \cong
  M^{\mathbf b''}(\lambda-\rho^{\mathbf b''}).
\]
Moreover, \(\mathbf b'\) and \(\mathbf b''\) have the same image in
\(L(m,n)_\lambda\).  Hence
\[
  \ell_{\lambda}(\mathbf b,\mathbf b')
  =
  \ell_{\lambda}(\mathbf b,\mathbf b''),
\]
and the claim follows from the induction hypothesis.
Now assume \((\lambda,\alpha)=0\).  Then
\[
  \ell_{\lambda}(\mathbf b,\mathbf b')+1
  =
  \ell_{\lambda}(\mathbf b,\mathbf b'').
\]
For the adjacent Borels \(\mathbf b'\) and \(\mathbf b''\), we have the exact
sequence
\[
  0
  \longrightarrow
  M^{\mathbf b'}(\lambda+\alpha-\rho^{\mathbf b'})
  \longrightarrow
  U^{\mathbf b'\subset \mathbf b'+\mathbf b''}
    (\lambda-\rho^{\mathbf b'})
  \longrightarrow
  M^{\mathbf b'}(\lambda-\rho^{\mathbf b'})
  \longrightarrow
  0,
\]
and
\[
  \operatorname{top}
  M^{\mathbf b'}(\lambda+\alpha-\rho^{\mathbf b'})
  \cong
  L^{\mathbf b''}(\lambda-\rho^{\mathbf b''}).
\]
Since \(P^{\mathbf b}(\lambda-\rho^{\mathbf b})\) has a
\(U^{\mathbf b'\subset \mathbf b'+\mathbf b''}\)-flag and contains
\(U^{\mathbf b'\subset \mathbf b'+\mathbf b''}(\lambda-\rho^{\mathbf b'})\)
as a subquotient, the above exact sequence shifts the corresponding
composition multiplicity by one degree.  Hence
\[
  \sum_{r\geq 0}
  \bigl[
    M^{\mathbf b}(\lambda-\rho^{\mathbf b})
    :
    L^{\mathbf b''}(\lambda-\rho^{\mathbf b''})\langle -r\rangle
  \bigr]q^r
  =
  q\,
  q^{\ell_{\lambda}(\mathbf b,\mathbf b')}
  =
  q^{\ell_{\lambda}(\mathbf b,\mathbf b'')}.
\]
This completes the induction.
\end{proof}

\subsection{Some classes of dominant weights}
\label{subsec:dominant-weight-classes}

\subsubsection{Closed formula for maximally singular weights}
\label{subsubsec:maximally-singular-weights}

Throughout this subsection, we fix
$\mathbf b=0^m1^n\in L(m,n)$.

\begin{theorem}\label{thm:closed-form-e-r}
Let $p=(1^m|0^n)$ and let $0\leq r\leq m+n$.
For $S\subseteq\{1,\ldots,m+n\}$, let $p^S$ be obtained from $p$ by
interchanging $0$ and $1$ precisely at the positions in $S$. Then

\[
e^{(r)}v_p
=
\sum_{\substack{S\subseteq\{1,\ldots,m+n\}\\ |S|=r}}
q^{\sum_{i\in S}i-\frac{r(r+1)}2}
v_{p^S}.
\]
\end{theorem}

\begin{proof}
This is a direct concequence of the coproduct formula.
See also examples\cref{app:gl22-examples,app:gl33-examples}.
\end{proof}

\begin{remark}
The formula above can also be reformulated in terms of lengths of paths in
$L(m,n)$. However, the computations in
\cref{subsec:gl22-s2-kl} suggest that, for other choices of $\mathbf b$, one
should not expect a formula of comparably simple form.
\end{remark}

\subsubsection{Characterizations of Kac--Wakimoto weights}
\label{subsubsec:kac-wakimoto-weights}

A weight satisfying the equivalent conditions below is called a
\emph{Kac-Wakimoto weight} (a.k.a. Kostant, totally connected, tame weight).

The equivalence of \((2)\), \((3)\), and \((4)\) was proved in
\cite{brundanstroppelII}, \cite{heidersdorf2017mixed}, and \cite{chmutov2015kac}, respectively. The conditions \((5)\), \((6)\),
and \((7)\) appear to be new. Note that \((7)\) is an invariant of the abelian
category \(\mathcal O\).
Condition \((8)\), due to Kac and Wakimoto \cite{kac1994integrable}, appears to be the earliest
definition.

\begin{theorem}\label{kww}
Let \(\mu\in\Lambda^{\mathrm{reg,dom}}\). Then the following conditions are
equivalent.
\begin{enumerate}
  \item
  There is no \(\wedge\) between two \(\vee\)'s in the weight diagram of
  \(\mu\).

  \item
  For every \(\lambda\in\Lambda^{\mathrm{reg,dom}}\), one has
  \[
    \sum_{i\geq0}
    \dim\operatorname{Ext}^i_{K_{\Lambda}}
    \bigl(K(\lambda),L(\mu)\bigr)
    \leq 1 .
  \]

  \item
  The simple module \(L(\mu)\) is isomorphic to a tensor product of an
  irreducible mixed tensor module with some Berezinian powers.

  \item
  The character of \(L(\mu)\) is given by the Kac--Wakimoto character formula.

  \item
  The set
  \[
    \bigl\{
      L^{\mathbf b}
      \bigl(\mu+\rho^{0^m1^n}-\rho^{\mathbf b}\bigr)
      \mid
      \mathbf b\in L(m,n)
    \bigr\}
  \]
  is precisely the set of composition factors of \(K(\mu)\).

  \item
  \[
    \ell K(\mu)=\operatorname{aty}(\mu)+1.
  \]

  \item
  \[
    \ell\bigl(\operatorname{rad}P_{\mathcal F}(\mu)/
    \operatorname{rad}^2P_{\mathcal F}(\mu)\bigr)
    =
    \operatorname{aty}(\mu)+1.
  \]
  \item There exist \(\mathbf b\in L(m,n)\), a weight \(\nu\), and
\(\operatorname{aty}L(\mu)\) mutually orthogonal \(\mathbf b\)-simple roots
\(\alpha_1,\ldots,\alpha_{\operatorname{aty}L(\mu)}\) such that
\(L(\mu)\cong L^{\mathbf b}(\nu)\) and
\[
  (\nu+\rho^{\mathbf b},\alpha_i)=0
  \qquad
  \text{for all }i.
\]
\end{enumerate}
\end{theorem}

\begin{proof}

The equivalence of \((5)\) and \((6)\) follows from
\cref{Fhw}. Indeed, note that for every regular dominant weight \(\mu\), the
set
\[
  \bigl\{
    M^{\mathbf b}
    \bigl(\mu+\rho^{0^m1^n}-\rho^{\mathbf b}\bigr)
    \mid
    \mathbf b\in L(m,n)
  \bigr\}
\]
has cardinality \(\operatorname{aty}(\mu)+1\).

It remains to prove \((7)\). Projective module
\(P_{\mathcal F}(\mu)\) is rigid. Hence it is enough to count the composition
factors in degree \(1\). By \cref{Fhw}, this number is
\[
  \#\{\lambda\mid \deg(\mu\overline{\lambda})=1\}
  +
  \#\{\lambda\mid \deg(\underline{\lambda}\mu)=1\}.
\]
The first term is precisely \(\operatorname{aty}(\mu)\). The second term is
minimal, and equal to \(1\), exactly in the case of condition \((1)\). Thus
\[
  \ell\bigl(\operatorname{rad}P_{\mathcal F}(\mu)/
  \operatorname{rad}^2P_{\mathcal F}(\mu)\bigr)
  =
  \operatorname{aty}(\mu)+1
\]
if and only if \((1)\) holds. This proves the claim.
\end{proof}

\subsection{Results of Brundan--Goodwin}
\label{subsec:brundan-goodwin}
In this subsection, we review the work of Brundan and Goodwin
\cite{brundan2019whittaker}.

Let $\mathfrak g=\mathfrak{gl}(n|n)$ and fix the principal nilpotent element
\[
e
:=
\sum_{i=1}^{n-1}e_{i,i+1}
+
\sum_{i=1}^{n-1}e_{n+i,n+i+1}
\in\mathfrak g_{\bar0}.
\]
Define $\chi\in\mathfrak g^*$ by
$\chi(x):=\operatorname{str}(ex)$ for $x\in\mathfrak g$.

Let $\mathfrak p$ and $\mathfrak m$ be the subalgebras associated with the
principal good grading, and set
$\mathfrak m_\chi:=\{x-\chi(x)\mid x\in\mathfrak m\}\subset U(\mathfrak g)$.
Let $\mathfrak g\text{-sMod}_\chi$ be the full subcategory of
$\mathfrak g\text{-sMod}$ consisting of supermodules on which
$\mathfrak m_\chi$ acts locally nilpotently.

Following the right-module convention of Brundan--Goodwin, set
\[
Q_\chi
:=
U(\mathfrak g)/\mathfrak m_\chi U(\mathfrak g),
\qquad
U(\mathfrak g,e)
:=
\operatorname{End}_{\text{sMod-}U(\mathfrak g)}(Q_\chi).
\]
Thus $Q_\chi$ is naturally a
$(U(\mathfrak g,e),U(\mathfrak g))$-superbimodule.

The Whittaker coinvariant construction equips
$M/\mathfrak m_\chi M$ with a natural $U(\mathfrak g,e)$-module structure.
We write
\[
H_0(M):=M/\mathfrak m_\chi M.
\]

\begin{lemma}[\cite{brundan2019whittaker}]
\label{lem:H0_welldefined_exact_duality}
The Whittaker coinvariants functor
\[
H_0:s\mathcal O
\longrightarrow
U(\mathfrak g,e)\text{-sMod}_{\mathrm{fd}},
\qquad
M\longmapsto H_0(M),
\]
is well defined and exact.
\end{lemma}

\begin{theorem}[\cite{brown2013principal}]
\label{thm:Miura_injective}
The Miura transform is an injective algebra homomorphism
\[
U(\mathfrak g,e)
\longrightarrow
U(\mathfrak p)
\longrightarrow
U(\mathfrak g(0))
\cong
U\bigl(\mathfrak{gl}(1|1)^{\oplus n}\bigr),
\]
where $\mathfrak g(0)$ is the degree-zero part of the principal good
grading.
\end{theorem}

Hence every $\mathfrak{gl}(1|1)^{\oplus n}$-module may be regarded as a
$U(\mathfrak g,e)$-module via the Miura transform.

\begin{theorem}[\cite{brundan2019whittaker}]
\label{thm:H0BG_identity}
For every $M\in\mathfrak{gl}(1|1)^{\oplus n}\text{-sMod}$, there is a
natural $U(\mathfrak g,e)$-module isomorphism
\[
H_0(\operatorname{BG}(M))\cong M,
\]
where $M$ on the right-hand side is regarded as a $U(\mathfrak g,e)$-module
via the Miura transform.
\end{theorem}

\begin{definition}\label{def:BG_modules}
Define
\[
\Lambda^{\mathrm{BG}}
:=
\left\{
\lambda\in\Lambda
\ \middle|\
\operatorname{aty}L^{0^n1^n}(\lambda)
=
\sum_{i=1}^n
\operatorname{aty}L(\lambda_i\mid\lambda_{n+i})
\right\}.
\]
\end{definition}

\begin{theorem}[\cite{brown2013principal,brundan2019whittaker}]
\label{thm:Uge_Verma_modules}
The following statements hold.
\begin{enumerate}
\item
For every $\lambda\in\Lambda$,
\[
H_0M^{()}(\lambda-\rho^{()})
\cong
\overline M(\lambda),
\]
where $\overline M(\lambda)$ is the Verma module for $U(\mathfrak g,e)$
defined using the triangular decomposition arising from the super Yangian
realization. In particular, its isomorphism class depends only on the Weyl group dot orbit
of $\lambda$.

\item
The module
$\overline L(\lambda):=\operatorname{top}\overline M(\lambda)$ is simple,
and every simple $U(\mathfrak g,e)$-supermodule is isomorphic to
$\overline L(\lambda)$ for some $\lambda\in\Lambda$.

\item
We have
$H_0N^{s^n}(\lambda)\cong\overline L(\lambda)$ if and only if
$\lambda\in\Lambda^{\mathrm{BG}}$.

\item
We have
\[
H_0L^{()}(\lambda-\rho^{()})
\cong
\begin{cases}
\overline L(\lambda),
& \lambda-\rho^{()}\text{ is antidominant},\\
0,
& \text{otherwise}.
\end{cases}
\]
\end{enumerate}
\end{theorem}

\begin{definition}\label{def:pi_Serre}
Let
$\mathbb V:s\mathcal O\to s\bar{\mathcal O}$ and
$\mathbb V:\mathcal O\to\bar{\mathcal O}$
denote the Soergel functors. They are the Serre quotient functors by the
Serre subcategories generated by the non-antidominant simple objects.
\end{definition}

\begin{remark}\label{rem:barO_properly_stratified}
Chen--Cheng--Mazorchuk \cite{chen2023whittaker} showed that
$\bar{\mathcal O}$ is a properly stratified category.
\end{remark}

\begin{theorem}[\cite{brown2013principal,brundan2019whittaker}]
\label{thm:irr_barO}
There is an equivalence of categories
\[
\kappa:H_0(s\mathcal O)\xrightarrow{\sim}s\bar{\mathcal O}
\]
such that $\mathbb V\cong\kappa\circ H_0$ on $s\mathcal O$.
\end{theorem}

\begin{theorem}\label{thm:socle_simple}
If $\lambda\in\Lambda^{\mathrm{BG}}$, then
\[
\operatorname{soc}N^{s^n}(\lambda)
\cong
L^{0^n1^n}\bigl((\lambda-\rho)^{\mathrm{antidom}}\bigr).
\]
\end{theorem}

\begin{proof}
By \cref{thm:irr_barO}, the module $N^{s^n}(\lambda)$ has a unique
antidominant composition factor. On the other hand,
$N^{s^n}(\lambda)\in\mathcal F M_{\bar0}$, so
\cref{antidom} implies that every simple submodule of
$N^{s^n}(\lambda)$ is antidominant. Hence its socle is simple.
The assertion now follows from parts~(2) and~(4) of
\cref{thm:Uge_Verma_modules}.
\end{proof}

\begin{theorem}\label{thm:antidom_factors}
Let $\mathbf b\in L(n,n)$ and $\lambda\in\Lambda$. Choose
$\lambda^{\mathrm{BG}}\in(W\cdot\lambda)\cap\Lambda^{\mathrm{BG}}$. Then
\[
\bigoplus_{\nu\in\Lambda^{\mathrm{antidom}}}
L^{0^n1^n}(\nu)^{
\oplus[M^{\mathbf b}(\lambda-\rho^{\mathbf b}):
L^{0^n1^n}(\nu)]}
\cong
\bigoplus_{\mu\in\Lambda}
L^{0^n1^n}\bigl((\mu-\rho)^{\mathrm{antidom}}\bigr)^{
\oplus[\overline K(\lambda^{\mathrm{BG}}):\overline L(\mu)]}.
\]
Here $\overline K(\lambda)$ denotes the restriction, via the Miura
transform, of the corresponding Kac module for
$\mathfrak{gl}(1|1)^{\oplus n}$.
\end{theorem}

\begin{proof}
In the Grothendieck group of $\bar{\mathcal O}$, we have
\[
\begin{aligned}
[\mathbb V M^{\mathbf b}(\lambda-\rho^{\mathbf b})]
&=
[\mathbb V M^{()}(\lambda-\rho)]\\
&=
[\mathbb V M^{()}(\lambda^{\mathrm{BG}}-\rho)]\\
&=
[\mathbb V M^{(01)^n}(\lambda^{\mathrm{BG}})]\\
&=
\sum_{\mu\in\Lambda}
[M^{(01)^n}(\lambda^{\mathrm{BG}}):N^{s^n}(\mu)]
[\mathbb V N^{s^n}(\mu)].
\end{aligned}
\]
By \cref{thm:H0BG_identity,thm:irr_barO},
\[
[M^{(01)^n}(\lambda^{\mathrm{BG}}):N^{s^n}(\mu)]
=
[\overline K(\lambda^{\mathrm{BG}}):\overline L(\mu)],
\]
while \cref{thm:socle_simple} gives
$\mathbb V N^{s^n}(\mu)
\cong
\mathbb V L^{0^n1^n}((\mu-\rho)^{\mathrm{antidom}})$.
Comparing the antidominant composition factors gives the result.
\end{proof}

The socle of a $\mathbf b$-Verma supermodule is known for
$\mathbf b=()$ and $\mathbf b=(01)^n$ by \cite{chen2021translated},
but remains unknown for general $\mathbf b$. The following consequence
of the preceding theorem will be useful for this problem.

\begin{corollary}\label{cor:antidom_factors_as_socles}
Let $\mathbf b\in L(n,n)$ and $\lambda\in\Lambda$. Choose
$\lambda^{\mathrm{BG}}\in(W\cdot\lambda)\cap\Lambda^{\mathrm{BG}}$. Then
$\operatorname{soc}M^{\mathbf b}(\lambda-\rho^{\mathbf b})$ is a direct
summand of
\[
\bigoplus_{\mu\in\Lambda}
L^{0^n1^n}\bigl((\mu-\rho)^{\mathrm{antidom}}\bigr)^{
\oplus[\overline K(\lambda^{\mathrm{BG}}):\overline L(\mu)]}.
\]
\end{corollary}

\begin{proof}
This follows immediately from \cref{thm:antidom_factors} and Irving's
criterion in \cref{antidom}.
\end{proof}

\subsection{Remarks on the classification of Morita equivalence classes of blocks}
\label{subsec:morita-classification-blocks}
In this subsection, we assume that $\mathbf b=0^m1^n$ and omit the
superscript $\mathbf b$ whenever no confusion can arise.

\begin{theorem}[\cite{coulembier2017homological}]
\label{thm:injective-projective-dimension}
Let $\lambda\in\Lambda$. Then
\[
\operatorname{pd}_{\mathcal O}I(\lambda)
=
\operatorname{pd}_{\mathcal O_{\bar0}}I_{\bar0}(\lambda).
\]
In particular,
\[
\lambda\text{ is regular dominant}
\quad\Longleftrightarrow\quad
\operatorname{pd}_{\mathcal O}I(\lambda)=2\ell(w_0),
\]
whereas
\[
\lambda\text{ is antidominant}
\quad\Longleftrightarrow\quad
\operatorname{pd}_{\mathcal O}I(\lambda)=0.
\]
\end{theorem}

Thus regular dominant and antidominant simple modules can be characterized
purely in terms of the underlying abelian category. In particular, the full
subcategory $\mathcal F$ of finite-dimensional modules is categorically
determined and hence is preserved under equivalences of blocks.

By \cref{kww}, the same is true for Kac--Wakimoto weights: the set of
Kac--Wakimoto weights in a block can be characterized purely in terms of
the abelian category. This provides a substantially smaller and more
structured family of objects on which one can test equivalence of blocks.

There is another useful categorical invariant coming from the results of
Brundan--Goodwin. Let $\mathcal B$ be a block of atypicality $r$.
By \cref{thm:Uge_Verma_modules} and
\cite{brundan2019whittaker}, for every Verma module $M(\lambda)$ in
$\mathcal B$ one has
\[
\sum_{\nu\in\Lambda^{\mathrm{antidom}}}
[M(\lambda):L(\nu)]
=
2^r.
\]
Indeed, the left-hand side is the composition length of the corresponding
Verma module for the principal finite $W$-superalgebra under the Whittaker
coinvariants functor.

For a dominant weight $\lambda$ in $\mathcal B$, define
\[
a_{\mathcal B}(\lambda)
:=
\sum_{\nu\in\Lambda^{\mathrm{antidom}}}
[P(\lambda):L(\nu)].
\]
By BGG reciprocity,
\[
\begin{aligned}
a_{\mathcal B}(\lambda)
&=
\sum_{\mu}
(P(\lambda):M(\mu))
\sum_{\nu\in\Lambda^{\mathrm{antidom}}}
[M(\mu):L(\nu)]\\
&=
2^r
\sum_{\mu}(P(\lambda):M(\mu))^2.
\end{aligned}
\]

The Verma flags of dominant projective modules can be computed by
Brundan's bumping algorithm. Consequently,
$a_{\mathcal B}(\lambda)$ is explicitly computable.

Considering all dominant projective modules at once, however, produces
far too much data. We therefore propose to restrict attention to those
dominant projectives indexed by Kac--Wakimoto weights. Since both the
Kac--Wakimoto condition and the antidominant simple objects admit
categorical characterizations, the resulting collection of numbers
$a_{\mathcal B}(\lambda)$ is an invariant of the underlying abelian
category. Moreover, once the Kac--Wakimoto weights in a block are arranged
according to the categorical ordering, these
numbers form a naturally ordered sequence of invariants.

This approach may be viewed as a refinement of the method of
Coulembier--Serganova \cite{coulembier2017homological}. They used
homological invariants to show, in particular, that category $\mathcal O$
for $\mathfrak{gl}(3|1)$ contains infinitely many pairwise non-equivalent
blocks. Our proposal is to supplement such homological invariants with the
multiplicities attached to the categorically distinguished
Kac--Wakimoto weights. Together with computer calculations, this may
provide new examples toward the classification of Morita equivalence
classes of blocks in super category $\mathcal O$.

The results of Brundan--Goodwin
\cite{brundan2019whittaker} already impose strong restrictions on Morita
equivalences between blocks. For $\mathfrak{gl}(n|n)$, the case of
atypicality $n-1$ is essentially explicit from their results. In contrast,
the classification of blocks of atypicality $n-2$ appears to remain open
in general. Thus this is a natural first case in which to test the
invariants proposed above.

Homological dimensions provide additional invariants that should also be
taken into account. In particular, the finitistic projective dimension of
a block is preserved under equivalence and gives a basic numerical
constraint complementary to the invariants above.

There is also a conceptually different approach to the classification
problem, suggested by \cref{conj:highest-weight-structures-from-borels}.
This conjecture is motivated by Coulembier's classification of blocks in
the classical BGG category $\mathcal O$
\cite{coulembier2020classification}. A key ingredient in the classical
setting is that, in the presence of a simple-preserving duality, the
highest weight structure of the relevant finite highest weight category is
unique up to equivalence. The conjecture above may be regarded as a super
analogue in which uniqueness is replaced by the finite family of highest
weight structures arising from the different Borel subalgebras
$\mathbf b\in L(m,n)$.

\subsection{Parabolic Miura transforms}
\label{subsec:parabolic-miura-transforms}

\begin{theorem}[Parabolic Miura transform {\cite{brown2013principal}}]\label{thm:parabolic_Miura}
Fix a composition $\mu=(\mu_1,\dots,\mu_r)$ of $n$ and consider the Levi subsuperalgebra
\[
\mathfrak l_{\mu}:=\bigoplus_{i=1}^{r}\mathfrak{gl}(\mu_i|\mu_i)\ \subset\ \mathfrak{gl}(n|n).
\]
Let $\mathfrak p_{\mu}\subset \mathfrak{gl}(n|n)$ be the parabolic subsuperalgebra coming from the principal
good grading with Levi factor $\mathfrak l_{\mu}$, and let $e_{\mu}\in(\mathfrak l_{\mu})_{\bar0}$ be the
principal nilpotent element (equivalently, the image of $e$ under the projection $\mathfrak p_{\mu}\twoheadrightarrow
\mathfrak l_{\mu}$).

The \emph{parabolic Miura transform} is the algebra homomorphism
\[
\mathrm{MT}_{\mu}:\ U(\mathfrak{gl}(n|n),e)\ \longrightarrow\ U(\mathfrak l_{\mu},e_{\mu}),
\]
defined as the composite
\[
U(\mathfrak{gl}(n|n),e)\ \longrightarrow\ U(\mathfrak p_{pr})\ \twoheadrightarrow\ U(\mathfrak p_{\mu}),
\]
followed by the natural projection $U(\mathfrak p_{\mu})\twoheadrightarrow U(\mathfrak l_{\mu})$ and the
canonical quotient map $U(\mathfrak l_{\mu})\twoheadrightarrow U(\mathfrak l_{\mu},e_{\mu})$.

Moreover, $\mathrm{MT}_{\mu}$ is injective.
\end{theorem}

\begin{theorem}\label{thm:MT_inverse_on_Verma_nogamma}
Fix a composition $\mu=(\mu_1,\dots,\mu_r)$ of $n$ and a choice
$c=(c_1,\dots,c_r)\in\{0,1\}^r$.
Set
\[
\mathfrak l_\mu:=\bigoplus_{i=1}^r \mathfrak{gl}(\mu_i|\mu_i)\subset \mathfrak{gl}(n|n),
\qquad
\mu_{<i}:=\sum_{j<i}\mu_j\ (1\le i\le r).
\]
For $\lambda\in\Lambda$, write
\[
\lambda|_{\mathfrak{gl}(\mu_i|\mu_i)}\in \Lambda_{\mu_i|\mu_i}
\]
for the restriction of $\lambda$ to the Cartan subalgebra of the $i$-th Levi summand
$\mathfrak{gl}(\mu_i|\mu_i)\subset \mathfrak l_\mu$.
Let $\overline{M}_i(\lambda)$ denote the Verma $U(\mathfrak{gl}(\mu_i|\mu_i),e_{\mu_i})$-supermodule of highest
weight $\lambda|_{\mathfrak{gl}(\mu_i|\mu_i)}$ (defined using the standard triangular decomposition coming from
the super Yangian realization), and set
\[
\overline{M}_i^{c_i}(\lambda):=
\begin{cases}
\overline{M}_i(\lambda), & c_i=0,\\
\overline{M}_i(\lambda)^{\vee}, & c_i=1.
\end{cases}
\]
Using the canonical identification
$U(\mathfrak l_\mu,e_\mu)\cong \bigotimes_{i=1}^r U(\mathfrak{gl}(\mu_i|\mu_i),e_{\mu_i})$, define the
$U(\mathfrak l_\mu,e_\mu)$-module
\[
\overline{M}^{c}(\lambda)
:=\overline{M}_1^{c_1}(\lambda)\boxtimes\cdots\boxtimes \overline{M}_r^{c_r}(\lambda).
\]

Then the pullback of $\overline{M}^{c}(\lambda)$ along $\mathrm{MT}_\mu$ is isomorphic to the Whittaker
coinvariants of a Verma module:
\[
\mathrm{MT}_\mu^{*}\bigl(\overline{M}^{c}(\lambda)\bigr)
\ \cong\
H_{0}\,M^{\mathfrak b_{\mu,c}}\!\bigl(\lambda-\rho^{\mathfrak b_{\mu,c}}\bigr),
\]
where $\mathfrak b_{\mu,c}\in L(n,n)$ is the Borel subalgebra corresponding to the partition
\[
\bigl((\mu_{<r}+c_r\mu_r)^{\mu_r},\ (\mu_{<r-1}+c_{r-1}\mu_{r-1})^{\mu_{r-1}},\ \dots,\ (c_1\mu_1)^{\mu_1}\bigr).
\]

\end{theorem}

\begin{proof}
It suffices to show that
\[
\mathrm{MT}_{\mu}^{*}\circ H_{0}^{\mathfrak l_{\mu}}
\ \cong\
H_{0}^{\mathfrak g}\circ
\operatorname{Ind}^{\mathfrak g}_{\mathfrak p_{\mu}}\circ
\operatorname{Infl}^{\mathfrak p_{\mu}}_{\mathfrak l_{\mu}},
\]
where $H_{0}^{\mathfrak g}$ (resp.\ $H_{0}^{\mathfrak l_{\mu}}$) denotes the Whittaker coinvariants functor for
$\mathfrak g$ (resp.\ $\mathfrak l_{\mu}$).

Write $\mathfrak u_{\mu}$ for the positive nilradical of $\mathfrak p_{\mu}$.
The key point is the following natural bimodule isomorphism:
\[
\bigl(U(\mathfrak g)/\mathfrak m_{\chi}U(\mathfrak g)\bigr)\Big/\bigl(U(\mathfrak g)/\mathfrak m_{\chi}U(\mathfrak g)\bigr)\,U(\mathfrak u_{\mu})
\ \cong\
U(\mathfrak l_{\mu})\Big/\bigl(\mathfrak m_{\chi}\cap \mathfrak l_{\mu}\bigr)\,U(\mathfrak l_{\mu}),
\]
as $\bigl(U(\mathfrak l_{\mu}),U(\mathfrak g)\bigr)$-superbimodules.

Indeed, using the vector space decompositions
\[
U(\mathfrak g)=\mathfrak m_{\chi}U(\mathfrak g)\ \oplus\ U(\mathfrak p),
\qquad
U(\mathfrak l_{\mu})=(\mathfrak m_{\chi}\cap\mathfrak l_{\mu})U(\mathfrak l_{\mu})\ \oplus\ U(\mathfrak p_{\mu}),
\]
we obtain a natural surjection
\[
U(\mathfrak p)\twoheadrightarrow
U(\mathfrak p_{\mu}),
\]
whose kernel is $U(\mathfrak p_{\mu})U(\mathfrak u_{\mu})$.
Passing to the quotient $U(\mathfrak g)/\mathfrak m_{\chi}U(\mathfrak g)$ identifies
\[
\bigl(U(\mathfrak g)/\mathfrak m_{\chi}U(\mathfrak g)\bigr)\Big/\bigl(U(\mathfrak g)/\mathfrak m_{\chi}U(\mathfrak g)\bigr)\,U(\mathfrak u_{\mu})
\ \cong\
U(\mathfrak l_{\mu})/(\mathfrak m_{\chi}\cap \mathfrak l_{\mu})U(\mathfrak l_{\mu}),
\]
and this isomorphism is natural and  bimodule isomorphism.
\end{proof}

\begin{remark}\label{rem:H0_Verma_Miura_limits}
For a general Borel subalgebra $\mathfrak b$, the module $H_{0}M^{\mathfrak b}(\lambda)$ cannot always be
interpreted as a $W$-superalgebra Verma module of the form appearing in
Theorem~\ref{thm:MT_inverse_on_Verma_nogamma}.

In small ranks, however, such an interpretation is often possible. For $\mathfrak{gl}(1|1)$ and
$\mathfrak{gl}(2|2)$ it is possible for every $\mathfrak b\in L(n,n)$. For $\mathfrak{gl}(3|3)$ it is possible
for every $\mathfrak b\in L(3,3)$ except $\mathfrak b=(1)$ and $\mathfrak b=(32)$. For $\mathfrak{gl}(4|4)$ it
is possible for every $\mathfrak b\in L(4,4)$ except
\[
(1),\ (2),\ (1^{2}),\ (4^{2}2),\ (43^{2}),\ (4^{2}3).
\]
\end{remark}

\appendix

\section{$\mathfrak{gl}(2|2)$ examples}
\label{app:gl22-examples}

\subsection{Computation of $0s1$-Kazhdan--Lusztig polynomials for dominant weights}
\label{subsec:gl22-0s1-kl}

In this subsection, we assume that \(\mathbf b=0011\) and \(\mathbf c=0s1\).
Let \(\mathcal O_\lambda\) denote the block containing \(L(\lambda)\), and let
\(\operatorname{pr}_\lambda\) denote the projection onto this block.
The induced module \(\operatorname{Ind} P_{\overline{0}}(\lambda)\) has a
Verma flag whose subquotients are, with multiplicities,
\[
  \bigl\{
    M(\lambda+\sum_{\alpha\in S}\alpha)
    \bigm|
    S\subseteq \Delta^+_{\overline{1}}
  \bigr\}.
\]
Since \(P(\lambda)\) is a direct summand of
\(\operatorname{pr}_\lambda\operatorname{Ind}P_{\overline{0}}(\lambda)\), the
Verma flag of \(P(\lambda)\) is obtained from a sub-multiset of this one.  This
makes the following posets natural.
\begin{definition}
For \(\lambda\), let \(I\langle\lambda\rangle\) be the multiset
\[
  \bigl\{
    \lambda+\sum_{\alpha\in S}\alpha
    \bigm|
    S\subseteq\Delta^+_{\overline{1}}
  \bigr\},
\]
ordered by the dominant order.  We regard it as a poset whose vertices are
labelled by their multiplicities.
Let \(R\langle\lambda\rangle\) be the subposet of \(I\langle\lambda\rangle\)
whose vertices are the highest weights appearing in a Verma flag of
\(\operatorname{pr}_\lambda\operatorname{Ind}P_{\overline{0}}(\lambda)\), again
counted with multiplicities.
Let \(r\langle\lambda\rangle\) be the subposet of \(R\langle\lambda\rangle\)
whose vertices are the highest weights appearing in a Verma flag of
\(P(\lambda)\), again counted with multiplicities.
\end{definition}

\subsubsection{Atypicality one}
\label{subsubsec:gl22-0s1-atypicality-one}
\medskip
\noindent
\textbf{Generic cases.}

Assume that $b,c\neq a,a+1$.

\medskip
\noindent
\textbf{Case $(ab|ac)$.}
\begin{align*}
e_a v_{(a+1,b|a,c)}
&=
v_{(a,b|a,c)}
+qv_{(a+1,b|a+1,c)}\\
&=
u_{(ab|ac)}
+qu_{(a+1,b|a+1,c)}.
\end{align*}

\medskip
\noindent
\textbf{Case $(ab|ca)$.}
\begin{align*}
e_a v_{(a+1,b|c,a)}
&=
v_{(a,b|c,a)}
+qv_{(a+1,b|c,a+1)}\\
&=
u_{(ab|ca)}
+qu_{(a+1,b|c,a+1)}.
\end{align*}

\medskip
\noindent
\textbf{Case $(ba|ac)$.}
\begin{align*}
e_a v_{(b,a+1|a,c)}
&=
v_{(b,a|a,c)}
+qv_{(b,a+1|a+1,c)}\\
&=
u_{(ba|ac)}.
\end{align*}

\medskip
\noindent
\textbf{Case $(ba|ca)$.}
\begin{align*}
e_a v_{(b,a+1|c,a)}
&=
v_{(b,a|c,a)}
+qv_{(b,a+1|c,a+1)}\\
&=
u_{(ba|ca)}
+qu_{(b,a+1|c,a+1)}.
\end{align*}

\medskip
\noindent
\textbf{Non-generic cases.}

By symmetry, it suffices to consider the following cases.

\medskip
\noindent
\textbf{Case $(20|00)$.}
\begin{align*}
e_0v_{(21|00)}
&=
v_{(20|00)}
+qv_{(21|10)}
+q^2v_{(21|01)}\\
&=
u_{(20|00)}
+q^2u_{(21|01)}.
\end{align*}

\medskip
\noindent
\textbf{Case $(10|00)$.}
\begin{align*}
e_1e_0v_{(21|00)}
&=
v_{(10|00)}
+qv_{(11|10)}
+qv_{(21|20)}\\
&\qquad
+q^2v_{(11|01)}
+q^2v_{(21|02)}\\
&=
u_{(10|00)}
+qu_{(21|20)}
+q^2u_{(11|01)}
+q^2u_{(21|02)}.
\end{align*}

\medskip
\noindent
\textbf{Case $(30|01)$.}
\begin{align*}
e_1v_{(32|01)}
&=
v_{(31|01)}
+qv_{(32|02)},\\
e_0e_1v_{(32|01)}
&=
v_{(30|01)}
+qv_{(31|11)}
+qv_{(32|12)}\\
&=
u_{(30|01)}
+qu_{(32|12)}.
\end{align*}

\medskip
\noindent
\textbf{Case $(20|01)$.}
\begin{align*}
e_2e_0e_1v_{(32|01)}
&=
v_{(20|01)}
+qv_{(21|11)}
+qv_{(22|12)}
+qv_{(32|13)}\\
&=
u_{(20|01)}
+qu_{(22|12)}
+qu_{(32|13)}.
\end{align*}

\medskip
\noindent
\textbf{Case $(21|01)$.}
\begin{align*}
e_1v_{(32|01)}
&=
v_{(31|01)}
+qv_{(32|02)},\\
e_2e_1v_{(32|01)}
&=
v_{(21|01)}
+qv_{(22|02)}
+qv_{(32|03)}\\
&=
u_{(21|01)}
+qu_{(22|02)}
+qu_{(32|03)}.
\end{align*}

\medskip
\noindent
\textbf{Case $(11|01)$.}
\begin{align*}
e_1^{(2)}v_{(22|01)}
&=
v_{(11|01)}
+qv_{(12|01)}
+q^2v_{(21|02)}\\
&=
u_{(11|01)}
+qu_{(12|01)}
+q^2u_{(21|02)}.
\end{align*}

In each case, $R\langle\lambda\rangle=r\langle\lambda\rangle$, and the corresponding Hasse diagram is one of the following:

\begin{figure}[ht]
\centering

\begin{minipage}{0.19\textwidth}
\centering
\begin{tikzpicture}[scale=0.55]
  \coordinate (A) at (3,3);
  \coordinate (B) at (1,2);
  \coordinate (C) at (5,2);
  \coordinate (D) at (1,1);
  \coordinate (E) at (5,1);
  \draw (A)--(B);
  \draw (A)--(C);
  \draw (B)--(D);
  \draw (C)--(E);
  \foreach \P in {A,B,C,D,E}{\fill (\P) circle (2pt);}
\end{tikzpicture}
\end{minipage}
\hfill
\begin{minipage}{0.19\textwidth}
\centering
\begin{tikzpicture}[scale=0.55]
  \coordinate (A) at (3,3);
  \coordinate (B) at (3,2);
  \coordinate (C) at (3,1);
  \coordinate (D) at (3,0);
  \draw (A)--(B);
  \draw (B)--(C);
  \draw (C)--(D);
  \foreach \P in {A,B,C,D}{\fill (\P) circle (2pt);}
\end{tikzpicture}
\end{minipage}
\hfill
\begin{minipage}{0.19\textwidth}
\centering
\begin{tikzpicture}[scale=0.55]
  \coordinate (A) at (3,3);
  \coordinate (B) at (1,2);
  \coordinate (C) at (5,2);
  \draw (A)--(B);
  \draw (A)--(C);
  \foreach \P in {A,B,C}{\fill (\P) circle (2pt);}
\end{tikzpicture}
\end{minipage}
\hfill
\begin{minipage}{0.19\textwidth}
\centering
\begin{tikzpicture}[scale=0.55]
  \coordinate (A) at (3,3);
  \coordinate (B) at (3,2);
  \coordinate (C) at (3,1);
  \draw (A)--(B);
  \draw (B)--(C);
  \foreach \P in {A,B,C}{\fill (\P) circle (2pt);}
\end{tikzpicture}
\end{minipage}
\hfill
\begin{minipage}{0.19\textwidth}
\centering
\begin{tikzpicture}[scale=0.55]
  \coordinate (A) at (3,3);
  \coordinate (B) at (3,2);
  \draw (A)--(B);
  \foreach \P in {A,B}{\fill (\P) circle (2pt);}
\end{tikzpicture}
\end{minipage}

\end{figure}

\subsubsection{Atypicality two}
\label{subsubsec:gl22-0s1-atypicality-two}
In this subsection, we determine $R\langle\lambda\rangle$ for each
$\lambda$.

\medskip
\noindent
\textbf{Case $(00|00)$.}
\begin{align*}
e_0^{(2)}v_{(11|00)}
&=
v_{(00|00)}
+qv_{(01|10)}
+q^2v_{(01|01)}
+q^2v_{(10|10)}\\
&\qquad
+q^3v_{(10|01)}
+q^4v_{(11|11)}\\
&=
u_{(00|00)}
+q^2u_{(01|01)}
+q^2u_{(10|10)}
+q^3u_{(10|01)}.
\end{align*}

\begin{center}
\begin{tikzpicture}[scale=0.85,every node/.style={transform shape}]
  \node (A) at (4,0) {$(22|22)$};

  \node (B) at (1,1) {$(21|12)$};
  \node (C) at (3,1) {$(21|21)$};
  \node (D) at (5,1) {$(12|12)$};
  \node (E) at (7,1) {$(12|21)$};

  \node (F) at (4,2) {$(11|11)_2$};

  \node (G) at (1,3) {$(01|10)$};
  \node (H) at (3,3) {$(01|01)$};
  \node (I) at (5,3) {$(10|10)$};
  \node (J) at (7,3) {$(10|01)$};

  \node (K) at (4,4) {$(00|00)$};

  \draw (A)--(B);
  \draw (A)--(C);
  \draw (A)--(D);
  \draw (A)--(E);

  \draw (B)--(F);
  \draw (C)--(F);
  \draw (D)--(F);
  \draw (E)--(F);

  \draw (F)--(G);
  \draw (F)--(H);
  \draw (F)--(I);
  \draw (F)--(J);

  \draw (G)--(K);
  \draw (H)--(K);
  \draw (I)--(K);
  \draw (J)--(K);
\end{tikzpicture}
\end{center}

Hence
\[
\operatorname{pr}_{\mathcal O^{ma}}
\operatorname{Ind}_{\mathfrak g_{\bar0}}^{\mathfrak g}
P_{\bar0}(00|00)
=
P^{0^21^2}(00|00)
\oplus
P^{0^21^2}(11|11).
\]

\medskip
\noindent
\textbf{Case $(10|01)$.}
\begin{align*}
e_1v_{(32|01)}
&=
v_{(31|01)}
+qv_{(32|02)},\\
e_0e_1v_{(32|01)}
&=
v_{(30|01)}
+qv_{(31|11)}
+qv_{(32|12)},\\
e_2e_0e_1v_{(32|01)}
&=
v_{(20|01)}
+qv_{(21|11)}
+qv_{(22|12)}
+qv_{(32|13)}.
\end{align*}

Applying $e_1$ once more gives
\begin{align*}
e_1e_2e_0e_1v_{(32|01)}
&=
v_{(10|01)}
+qv_{(20|02)}
+qv_{(11|11)}
+qv_{(21|21)}\\
&\qquad
+2q^2v_{(21|12)}
+qv_{(12|12)}
+q^3v_{(22|22)}\\
&\qquad
+qv_{(31|13)}
+q^2v_{(32|23)}\\
&=
u_{(10|01)}
+qu_{(20|02)}
+qu_{(21|21)}
+qu_{(12|12)}\\
&\qquad
+q^2u_{(21|12)}
+qu_{(31|13)}.
\end{align*}

\begin{center}
\begin{tikzpicture}[scale=0.9,every node/.style={transform shape}]
  \node (A) at (0,0) {$(32|23)$};
  \node (B) at (3,1) {$(31|13)$};
  \node (C) at (0,1) {$(22|22)$};

  \node (D) at (-1,2) {$(21|21)$};
  \node (E) at (1,2) {$(12|12)$};
  \node (F) at (3,2) {$(21|12)_2$};

  \node (G) at (0,3) {$(11|11)$};
  \node (H) at (3,3) {$(20|02)$};

  \node (I) at (0,4) {$(10|01)$};

  \draw (A)--(B);
  \draw (A)--(C);

  \draw (B)--(F);

  \draw (C)--(D);
  \draw (C)--(E);
  \draw (C)--(F);

  \draw (D)--(G);
  \draw (E)--(G);
  \draw (F)--(G);
  \draw (F)--(H);

  \draw (G)--(I);
  \draw (H)--(I);
\end{tikzpicture}
\end{center}

Hence
\[
\operatorname{pr}_{\mathcal O^{ma}}
\operatorname{Ind}_{\mathfrak g_{\bar0}}^{\mathfrak g}
P_{\bar0}(10|01)
=
P^{0^21^2}(10|01).
\]

\medskip
\noindent
\textbf{Case $(c0|0c)$, $c>1$.}
We first have
\[
e_0v_{(c+1,1|0,c)}
=
v_{(c+1,0|0,c)}
+qv_{(c+1,1|1,c)}.
\]
Applying $e_c$ gives
\begin{align*}
e_ce_0v_{(c+1,1|0,c)}
&=
v_{(c,0|0,c)}
+qv_{(c+1,0|0,c+1)}\\
&\qquad
+qv_{(c,1|1,c)}
+q^2v_{(c+1,1|1,c+1)}\\
&=
u_{(c0|0c)}
+qu_{(c+1,0|0,c+1)}.
\end{align*}

\begin{center}
\begin{tikzpicture}[scale=0.8,every node/.style={transform shape}]
  \node (A) at (4,0) {$(c+2,2\mid2,c+2)$};

  \node (B) at (1,1) {$(c+2,1\mid1,c+2)$};
  \node (C) at (7,1) {$(c+1,2\mid2,c+1)$};

  \node (D) at (4,2) {$(c+1,1\mid1,c+1)_2$};

  \node (E) at (1,3) {$(c+1,0\mid0,c+1)$};
  \node (F) at (7,3) {$(c,1\mid1,c)$};

  \node (G) at (4,4) {$(c,0\mid0,c)$};

  \draw (A)--(B);
  \draw (A)--(C);

  \draw (B)--(D);
  \draw (C)--(D);

  \draw (D)--(E);
  \draw (D)--(F);

  \draw (E)--(G);
  \draw (F)--(G);
\end{tikzpicture}
\end{center}

Hence
\[
\operatorname{pr}_{\mathcal O^{ma}}
\operatorname{Ind}_{\mathfrak g_{\bar0}}^{\mathfrak g}
P_{\bar0}(c0|0c)
=
P^{0^21^2}(c0|0c)
\oplus
P^{0^21^2}(c+1,0|0,c+1).
\]

\subsection{Computation of $s^2$-Kazhdan--Lusztig polynomials in atypicality two}
\label{subsec:gl22-s2-kl}
In this subsection, we assume that $\mathbf b=0101$ and
$\mathbf c=s^2$. We compute in the same manner as Cao and Lam
\cite{cao2016inversion}.

\medskip
\noindent
\textbf{Case $(0110)$.}
\begin{align*}
b_{(0110)}
&=
qv_{(0000)}
+v_{(0110)}
+q^2v_{(0011)}
+q^2v_{(1100)}\\
&\qquad
+q^4v_{(1001)}
+q^3v_{(1111)}\\
&=
qu_{(0000)}
+u_{(0110)}
+q^4u_{(1001)}.
\end{align*}

\medskip
\noindent
\textbf{Case $(1111)$.}
\begin{align*}
b_{(1111)}
&=
qv_{(1001)}
+q^2v_{(2002)}
+v_{(1111)}
+qv_{(2211)}\\
&\qquad
+q^3v_{(2112)}
+qv_{(2112)}
+qv_{(1122)}\\
&\qquad
+q^2v_{(2222)}
+q^2v_{(3113)}
+qv_{(3223)}\\
&=
qu_{(1001)}
+q^2u_{(2002)}
+u_{(1111)}
+q^3u_{(2112)}\\
&\qquad
+qu_{(2112)}
+q^2u_{(3113)}
+qu_{(3223)}.
\end{align*}

\medskip
\noindent
\textbf{Case $(c11c)$, $c>1$.}
\begin{align*}
b_{(c11c)}
&=
qv_{(c,0,0,c)}
+q^2v_{(c+1,0,0,c+1)}\\
&\qquad
+v_{(c,1,1,c)}
+qv_{(c+1,1,1,c+1)}\\
&=
qu_{(c,0,0,c)}
+q^2u_{(c+1,0,0,c+1)}\\
&\qquad
+u_{(c,1,1,c)}
+qu_{(c+1,1,1,c+1)}.
\end{align*}

\section{$\mathfrak{gl}(3|3)$ examples}
\label{app:gl33-examples}

\begin{example}\label{ex:e3-on-111000}

\[
\begin{aligned}
ev_{(111|000)}
&=
v_{(011|000)}
+qv_{(101|000)}
+q^2v_{(110|000)}
+q^3v_{(111|100)}
+q^4v_{(111|010)}
+q^5v_{(111|001)}.
\end{aligned}
\]

\[
\begin{aligned}
e^2v_{(111|000)}
&=
[2]\Bigl(
v_{(001|000)}
+qv_{(010|000)}
+q^2v_{(011|100)}
+q^3v_{(011|010)}
+q^4v_{(011|001)}
\\
&\quad
+q^2v_{(100|000)}
+q^3v_{(101|100)}
+q^4v_{(101|010)}
+q^5v_{(101|001)}
\\
&\quad
+q^4v_{(110|100)}
+q^5v_{(110|010)}
+q^6v_{(110|001)}
+q^6v_{(111|110)}
+q^7v_{(111|101)}
+q^8v_{(111|011)}
\Bigr).
\end{aligned}
\]

\[
\begin{aligned}
e^3v_{(111|000)}
&=
[3]!\Bigl(
v_{(000|000)}
+qv_{(001|100)}
+q^2v_{(001|010)}
+q^3v_{(001|001)}
\\
&\quad
+q^2v_{(010|100)}
+q^3v_{(010|010)}
+q^4v_{(010|001)}
+q^4v_{(011|110)}
\\
&\quad
+q^5v_{(011|101)}
+q^6v_{(011|011)}
+q^3v_{(100|100)}
+q^4v_{(100|010)}
\\
&\quad
+q^5v_{(100|001)}
+q^5v_{(101|110)}
+q^6v_{(101|101)}
+q^7v_{(101|011)}
\\
&\quad
+q^6v_{(110|110)}
+q^7v_{(110|101)}
+q^8v_{(110|011)}
+q^9v_{(111|111)}
\Bigr).
\end{aligned}
\]

Equivalently,
\[
e^3v_{(111|000)}
=
[3]!
\sum_{\substack{S\subset\{1,\ldots,6\}\\ |S|=3}}
q^{\sum_{i\in S}i-6}v_{(111|000)^S}.
\]
\end{example}

\begin{example}

\[
e_3v_{(654|123)}
=
v_{(653|123)}
+
qv_{(654|124)}.
\]

\[
e_2e_3v_{(654|123)}
=
v_{(652|123)}
+
qv_{(653|133)}
+
qv_{(654|134)}.
\]

\[
e_1e_2e_3v_{(654|123)}
=
v_{(651|123)}
+
qv_{(652|223)}
+
qv_{(653|233)}
+
qv_{(654|234)}.
\]

\[
e_4e_1e_2e_3v_{(654|123)}
=
v_{(641|123)}
+
qv_{(642|223)}
+
qv_{(643|233)}
+
qv_{(644|234)}
+
qv_{(654|235)}.
\]

\[
\begin{aligned}
e_3e_4e_1e_2e_3v_{(654|123)}
&=
v_{(631|123)}
+qv_{(641|124)}
+qv_{(632|223)}
+q^2v_{(642|224)} \\
&\quad
+qv_{(633|233)}
+qv_{(643|243)}
+2q^2v_{(643|234)}
+qv_{(634|234)} \\
&\quad
+q^3v_{(644|244)}
+qv_{(653|235)}
+q^2v_{(654|245)}.
\end{aligned}
\]

\[
\begin{aligned}
e_2e_3e_4e_1e_2e_3v_{(654|123)}
&=
v_{(621|123)}
+qv_{(622|223)}
+qv_{(623|233)}
+qv_{(624|234)} \\
&\quad
+qv_{(631|133)}
+2q^2v_{(632|233)}
+qv_{(632|323)}
+q^3v_{(633|333)} \\
&\quad
+q^2v_{(634|334)}
+qv_{(641|134)}
+3q^2v_{(642|234)} \\
&\quad
+qv_{(642|243)}
+qv_{(642|324)}
+2q^3v_{(643|334)}
+q^2v_{(643|343)} \\
&\quad
+q^3v_{(644|344)}
+qv_{(652|235)}
+q^2v_{(653|335)}
+q^2v_{(654|345)}.
\end{aligned}
\]

\[
\begin{aligned}
e_5e_2e_3e_4e_1e_2e_3v_{(654|123)}
&=
v_{(521|123)}
+qv_{(522|223)}
+qv_{(523|233)}
+qv_{(524|234)} \\
&\quad
+qv_{(531|133)}
+2q^2v_{(532|233)}
+qv_{(532|323)}
+q^3v_{(533|333)} \\
&\quad
+q^2v_{(534|334)}
+qv_{(541|134)}
+3q^2v_{(542|234)} \\
&\quad
+qv_{(542|243)}
+qv_{(542|324)}
+2q^3v_{(543|334)}
+q^2v_{(543|343)} \\
&\quad
+q^3v_{(544|344)}
+qv_{(552|235)}
+q^2v_{(553|335)}
+q^2v_{(554|345)} \\
&\quad
+qv_{(652|236)}
+q^2v_{(653|336)}
+q^2v_{(654|346)}.
\end{aligned}
\]

\[
\begin{aligned}
e_4e_5e_2e_3e_4e_1e_2e_3v_{(654|123)}
&=
v_{(421|123)}
+qv_{(422|223)}
+qv_{(423|233)}
+qv_{(424|234)} \\
&\quad
+qv_{(431|133)}
+2q^2v_{(432|233)}
+qv_{(432|323)}
+q^3v_{(433|333)} \\
&\quad
+q^2v_{(434|334)}
+qv_{(441|134)}
+3q^2v_{(442|234)} \\
&\quad
+qv_{(442|243)}
+qv_{(442|324)}
+2q^3v_{(443|334)}
+q^2v_{(443|343)} \\
&\quad
+q^3v_{(444|344)}
+qv_{(452|235)}
+q^2v_{(453|335)}
+q^2v_{(454|345)} \\
&\quad
+qv_{(524|235)}
+q^2v_{(534|335)}
+qv_{(541|135)} \\
&\quad
+4q^2v_{(542|235)}
+qv_{(542|253)}
+qv_{(542|325)} \\
&\quad
+3q^3v_{(543|335)}
+q^2v_{(543|353)}
+2q^3v_{(544|345)} \\
&\quad
+q^2v_{(544|354)}
+q^3v_{(554|355)}
+qv_{(642|236)} \\
&\quad
+q^2v_{(643|336)}
+q^2v_{(644|346)}
+q^2v_{(654|356)}.
\end{aligned}
\]

\[
\begin{aligned}
e_3e_4e_5e_2e_3e_4e_1e_2e_3v_{(654|123)}
&=
v_{(321|123)}
+qv_{(322|223)}
+qv_{(323|233)}
+qv_{(324|234)} \\
&\quad
+qv_{(331|133)}
+2q^2v_{(332|233)}
+qv_{(332|323)}
+q^3v_{(333|333)} \\
&\quad
+q^2v_{(334|334)}
+qv_{(341|134)}
+3q^2v_{(342|234)} \\
&\quad
+qv_{(342|243)}
+qv_{(342|324)}
+2q^3v_{(343|334)}
+q^2v_{(343|343)} \\
&\quad
+q^3v_{(344|344)}
+qv_{(352|235)}
+q^2v_{(353|335)}
+q^2v_{(354|345)} \\
&\quad
+qv_{(421|124)}
+q^2v_{(422|224)}
+2q^2v_{(423|234)}
+qv_{(423|243)} \\
&\quad
+q^3v_{(424|244)}
+2q^2v_{(431|134)}
+qv_{(431|143)}
+5q^3v_{(432|234)} \\
&\quad
+3q^2v_{(432|243)}
+2q^2v_{(432|324)}
+qv_{(432|423)} \\
&\quad
+(q^2+3q^4)v_{(433|334)}
+2q^3v_{(433|343)}
+q^2v_{(433|433)} \\
&\quad
+2q^4v_{(434|344)}
+q^3v_{(434|434)}
+q^3v_{(441|144)} \\
&\quad
+(q^2+3q^4)v_{(442|244)}
+q^3v_{(442|424)} \\
&\quad
+(q^3+3q^5)v_{(443|344)}
+2q^4v_{(443|434)}
+q^3v_{(443|443)} \\
&\quad
+q^6v_{(444|444)}
+q^2v_{(452|245)}
+2q^3v_{(453|345)}
+q^2v_{(453|435)} \\
&\quad
+q^4v_{(454|445)}
+qv_{(523|235)}
+q^2v_{(524|245)}
+qv_{(531|135)} \\
&\quad
+4q^2v_{(532|235)}
+qv_{(532|253)}
+qv_{(532|325)}
+(q+3q^3)v_{(533|335)} \\
&\quad
+q^2v_{(533|353)}
+3q^3v_{(534|345)}
+q^2v_{(534|354)}
+q^2v_{(534|435)} \\
&\quad
+q^2v_{(541|145)}
+4q^3v_{(542|245)}
+q^2v_{(542|254)}
+q^2v_{(542|425)} \\
&\quad
+5q^4v_{(543|345)}
+2q^3v_{(543|354)}
+3q^3v_{(543|435)}
+q^2v_{(543|453)} \\
&\quad
+2q^5v_{(544|445)}
+q^4v_{(544|454)}
+q^3v_{(553|355)}
+q^4v_{(554|455)} \\
&\quad
+qv_{(632|236)}
+q^2v_{(633|336)}
+q^2v_{(634|346)}
+q^2v_{(642|246)} \\
&\quad
+2q^3v_{(643|346)}
+q^2v_{(643|436)}
+q^4v_{(644|446)} \\
&\quad
+q^2v_{(653|356)}
+q^3v_{(654|456)}.
\end{aligned}
\]

\[
\begin{aligned}
e_3e_4e_5e_2e_3e_4e_1e_2e_3v_{(654|123)}
&=
u_{(321|123)}
+qu_{(331|133)}
+q^2u_{(332|233)}
+qu_{(341|134)} \\
&\quad
+2q^2u_{(342|234)}
+qu_{(342|243)}
+qu_{(352|235)}
+qu_{(421|124)} \\
&\quad
+2q^2u_{(431|134)}
+qu_{(431|143)}
+3q^3u_{(432|234)}
+2q^2u_{(432|243)} \\
&\quad
+q^2u_{(433|334)}
+q^3u_{(441|144)}
+(q^2+2q^4)u_{(442|244)}
+q^5u_{(443|344)} \\
&\quad
+q^2u_{(452|245)}
+q^3u_{(453|345)}
+qu_{(531|135)}
+3q^2u_{(532|235)} \\
&\quad
+qu_{(532|253)}
+qu_{(533|335)}
+q^2u_{(541|145)}
+3q^3u_{(542|245)} \\
&\quad
+q^2u_{(542|254)}
+2q^4u_{(543|345)}
+q^3u_{(543|354)}
+q^3u_{(553|355)} \\
&\quad
+qu_{(632|236)}
+q^2u_{(642|246)}
+q^3u_{(643|346)}
+q^2u_{(653|356)} \\
&\quad
+qv_{(323|233)}
+qv_{(324|234)}
+qv_{(332|323)}
+q^2v_{(334|334)} \\
&\quad
+qv_{(342|324)}
+q^3v_{(344|344)}
+q^2v_{(354|345)}
+2q^2v_{(423|234)} \\
&\quad
+qv_{(423|243)}
+q^3v_{(424|244)}
+2q^2v_{(432|324)}
+qv_{(432|423)} \\
&\quad
+q^2v_{(433|433)}
+2q^4v_{(434|344)}
+q^3v_{(442|424)}
+2q^4v_{(443|434)} \\
&\quad
+q^3v_{(443|443)}
+q^2v_{(453|435)}
+qv_{(523|235)}
+q^2v_{(524|245)} \\
&\quad
+qv_{(532|325)}
+3q^3v_{(534|345)}
+q^2v_{(534|354)}
+q^2v_{(542|425)} \\
&\quad
+3q^3v_{(543|435)}
+q^2v_{(543|453)}
+q^2v_{(634|346)}
+q^2v_{(643|436)}.
\end{aligned}
\]

\end{example}

\begin{example}
    \[
e_0v_{(531|024)}
=
v_{(530|024)}+qv_{(531|124)}.
\]

\[
e_2e_0v_{(531|024)}
=
v_{(520|024)}
+qv_{(531|034)}
+qv_{(521|124)}
+q^2v_{(531|134)}.
\]

\[
\begin{aligned}
e_4e_2e_0v_{(531|024)}
&=
v_{(420|024)}
+qv_{(530|025)}
+qv_{(430|034)}
+q^2v_{(530|035)} \\
&
+qv_{(421|124)}
+q^2v_{(531|125)}
+q^2v_{(431|134)}
+q^3v_{(531|135)}.
\end{aligned}
\]
\[
\begin{aligned}
&=
u_{(420|024)}
+qu_{(530|025)}
+qu_{(430|034)}
+q^2u_{(530|035)}
\end{aligned}
\]
\end{example}

\bibliographystyle{alpha}
\bibliography{refs}

\bigskip

\noindent
\textsc{Shunsuke Hirota} \\
\textsc{Department of Mathematics, Kyoto University} \\
Kitashirakawa Oiwake-cho, Sakyo-ku, 606-8502, Kyoto \\
\textit{E-mail address}: \href{mailto:hirota.shunsuke.48s@st.kyoto-u.ac.jp}{shun299509732@gmail.com}

\end{document}